\documentclass[12pt]{extarticle}
\usepackage[utf8]{inputenc}
\usepackage{xcolor}
\usepackage{verbatim}
\usepackage{float}
\usepackage{amsmath}
\usepackage{amsthm}
\usepackage{amssymb}
\usepackage{graphicx}
\usepackage{geometry}
\usepackage{wasysym}
\usepackage[pdfusetitle,
 bookmarks=true,bookmarksnumbered=false,bookmarksopen=false,
 breaklinks=false,pdfborder={0 0 0},pdfborderstyle={},backref=page,colorlinks=true]
 {hyperref}
\hypersetup{
 linkcolor=blue, citecolor=blue, urlcolor=blue}

\makeatletter
\numberwithin{equation}{section}
\numberwithin{figure}{section}

\@ifundefined{date}{}{\date{}}
\usepackage[utf8]{inputenc}
\usepackage{amsmath}
\usepackage{amssymb}
\usepackage{amsthm}
\usepackage{mathtools}
\usepackage{graphicx}
\usepackage{tikz-cd}
\usepackage{mlmodern}
\usepackage{fouriernc}
\usepackage[T1]{fontenc}

\usepackage{xcolor}
\color{black}
\usepackage{pdfrender}
\usepackage{etoolbox}
\AtBeginEnvironment{tikzpicture}{\pdfrender{TextRenderingMode=0}}
\AtBeginEnvironment{pgfpicture}{\pdfrender{TextRenderingMode=0}}
\usepackage{mathtools}        
\usepackage{bm}               
\usepackage{mathrsfs}         
\usepackage{dsfont}           
\usepackage{cancel}           
\usepackage{graphicx}
\usepackage{booktabs}

\newcommand\independent{\protect\mathpalette{\protect\independenT}{\perp}}
\def\independenT#1#2{\mathrel{\rlap{$#1#2$}\mkern2mu{#1#2}}}

\AtBeginDocument{

}

\makeatother

\theoremstyle{plain}
\newtheorem{thm}{\protect\theoremname}
\theoremstyle{definition}
\newtheorem{defn}[thm]{\protect\definitionname}
\theoremstyle{plain}
\newtheorem{prop}[thm]{\protect\propositionname}
\newtheorem{cor}[thm]{\protect\corollaryname}
\providecommand{\corollaryname}{Corollary}
\providecommand{\definitionname}{Definition}
\providecommand{\propositionname}{Proposition}
\providecommand{\theoremname}{Theorem}

\begin{document}
\title{Space-Time Transport of Brownian Exit Laws}
\author{Maher Boudabra \thanks{Preparatory Institute of Engineering Studies of Monastir, Tunisia. }}
\maketitle
\begin{abstract}
The present paper is devoted to a systematic study of the $p$-Brownian
convergence introduced in \cite{boudabra2026stability} (in press)
to study the stability of the planar Skorokhod embedding problem \cite{gross2019,Boudabra2020}.
The first part is an illustration of some geometric aspects of the
$p$-Brownian convergence. The second part turns this notion into
a metric between domains. More precisely, we place it within the framework
of optimal transport theory. Several results are obtained, namely
asymptotic behavior in case of homothetic domains. Numerical illustrations
are provided as well. 
\end{abstract}
\tableofcontents{}

\section{Preliminaries }

In a recent work \cite{boudabra2026stability}, motivated by the study
of numerical approximation of the solutions of the planar Skorokhod
embedding problem \cite{gross2019,BOUDABRA2026,Boudabra2020}, the
authors introduced a new mode of convergence for planar domains, which
they called \emph{$p$-Brownian convergence} with $p>0$. This mode
of convergence captures the geometry of domains through the time-space
exit pair of planar Brownian motion. We recall the definition in a
form that emphasizes that only a coupling of the exit-pair laws is
required.
\begin{defn}
\label{p-conv} Let $U$ be planar domain and let $(U_{n})_{n}$ be
a sequence of planar domains, all containing a common point, say the
origin. For each $n$, let $(Z^{[n]}_{t})_{t\geq0}$ be a planar Brownian
motion started at $0$ and denote by $\tau_{n}$ its exit time from
$U_{n}$. Similarly, let $(Z_{t})_{t\geq0}$ be a planar Brownian
motion started at $0$ and denote by $\tau$ its exit time from $U$.
We say that $(U_{n})_{n}$ converges to a domain $U$ in the \textbf{$p$-Brownian
sense} if there exist two random vectors $(\xi_{n},\theta_{n})$ and
$(\xi,\theta)$ on some common probability space, such that 
\[
(\xi_{n},\theta_{n})\sim(Z^{[n]}_{\tau_{n}},\tau_{n}),\,\,(\xi,\theta)\sim(Z_{\tau},\tau)
\]
and 
\[
\begin{alignedat}{1}\mathbf{E}\left(\big|\xi_{n}-\xi\big|^{p}\right)+\mathbf{E}\left(\big|\theta_{n}-\theta\big|^{\frac{p}{2}}\right) & \underset{n\to+\infty}{\longrightarrow}0\end{alignedat}
.
\]
\end{defn}

This formulation should be understood as a coupling of the exit-pair
laws only. It does not require the Brownian motions in $U_{n}$ and
$U$ to be realized as the same planar Brownian path. Same-path couplings
will appear later as a useful special case, but they are not part
of the definition. Note that $p$-Brownian convergence implies $q$-Brownian
convergence for $0<q<p$. The joint realization of the two random
vectors
\[
(\xi_{n},\theta_{n})\qquad\text{and}\qquad(\xi,\theta)
\]
is called a coupling of the exit-pair laws of
\[
(Z^{[n]}_{\tau_{n}},\tau_{n})\qquad\text{and}\qquad(Z_{\tau},\tau).
\]
 This will be elucidated in section \ref{sec:A-Brownian-Wasserstein-distance}.
The pair $(Z_{\tau_{U}},\tau_{U})$ consists of the Brownian exit
location and the Brownian exit time. The first component records where
the boundary is seen from the base point, while the second records
the time scale of escape. Taken together, these two variables provide
a natural probabilistic signature of $U$. In this sense, Brownian
convergence measures the effective geometry of a domain as seen by
a diffusing particle starting from the origin. The law of $Z_{\tau}$
is the harmonic measure $\omega^{0}_{U}$ of the domain $U$. If the
boundary of $U$ is smooth, then in many cases, one can manage to
compute the probability density $\rho$ of $Z_{\tau}$. In this context,
when $U$ is simply connected then univalent functions are widely
used to compute the p.d.f of $Z_{\tau}$. More precisely, if $f$
is a univalent map from $\mathbb{D}$ onto $U$ , with $f(0)=0$,
that extends to be analytic across $\partial\mathbb{D}$ then 
\begin{equation}
d\omega^{0}_{U}(z)=\frac{1}{2\pi}\bigl|(f^{-1})'(z)\bigr|\,|dz|,\qquad z\in\partial U.\label{density Z_tau}
\end{equation}
See \cite{markowsky2018distribution}. On the other hand, the law
of the exit time $\tau$ is usually much harder to compute explicitly.
To the best of our knowledge, there are only a few domains for which
the density of $\tau$ is explicit. The best-known example is the
half-plane, where the problem reduces to the hitting time of a one-dimensional
Brownian motion. Since the survival function $(z,t)\mapsto\mathbf{P}_{z}(\tau_{U}>t)$
solves the heat equation with Dirichlet boundary condition, PDE techniques
are often used to study the law of $\tau_{U}$. For example, one obtains
infinite-series representations when $U$ is the unit disc or an infinite
strip. The literature also contains estimates for more challenging
domains \cite{banuelos1997brownian,banuelos2005sharp}. Thus in our
framework, the joint law of the pair $(Z_{\tau},\tau)$ is hard to
compute explicitly. Moreover, $Z_{\tau}$ and $\tau$ are typically
not independent. 

\begin{defn}
Let $p>0$. The Hardy space of index $p$, denoted by $\mathbf{H}^{p}(\mathbb{D})$,
is the space of all analytic maps $f$ on the unit disc such that
\begin{equation}
\sup_{0<r<1}\frac{1}{2\pi}\int^{2\pi}_{0}|f(re^{i\theta})|^{p}\,d\theta<+\infty.\label{hardy norm def}
\end{equation}
The $p$-th root of \ref{hardy norm def} is denoted by $\Vert f\Vert_{\mathbf{H}^{p}(\mathbb{D})}$
and it is called the $p$-th Hardy norm of $f$ . 
\end{defn}

When $p<1$, the $\Vert\cdot\Vert_{\mathbf{H}^{p}(\mathbb{D})}$ is
not a ``genuine'' norm as the triangle inequality does not hold.
A Hardy norm encodes how the magnitude of $f$ varies on inner circles
of the unit disc. We refer the reader to \cite{Rudin2001,duren2000theory}
for more details about Hardy spaces. As the unit circle is of finite
measure, we have the inclusion 
\[
\mathbf{H}^{p+q}(\mathbb{D})\subset\mathbf{H}^{p}(\mathbb{D}).
\]
This inclusion raises a natural inquiry about the supremum $p$ for
which a map $f$ belongs to $\mathbf{H}^{p}(\mathbb{D})$.
\begin{defn}
Let $U\subsetneq\mathbb{C}$ be a simply connected domain, and let
$f:\mathbb{D}\to U$ be univalent onto $U$. The $\emph{Hardy number}$
of $U$ is defined by 
\[
h(U)=\sup\left\{ p>0\mid\Vert f\Vert_{\mathbf{H}^{p}(\mathbb{D})}<\infty\right\} .
\]
\end{defn}

The Hardy number $h(U)$ is a well defined: that is, it does not depend
on the choice of the conformal map $f$. It measures the largest exponent
$p$ for which a map $f$ from $\mathbb{D}$ onto $U$ belongs to
$\mathbf{H}^{p}(\mathbb{D})$. Hardy number goes back to the work
of Hansen \cite{hansen1970hardy}, who introduced this quantity in
his study of Hardy classes and ranges of analytic functions. Since
then, the Hardy number has been investigated from several geometric
points of view, including descriptions in terms of harmonic measure
and hyperbolic distance \cite{Karafyllia2019}. Its relevance to probability
was clarified by the seminal work of D. Burkholder \cite{burkholder1977exit},
who showed that for a simply connected planar domain $U$ the Hardy
number governs the moments of the Brownian exit time. More precisely,
if $\tau_{U}$ denotes the first exit time of planar Brownian motion
from $U$, then, among other results, Burkholder showed that if $f:\mathbb{D}\to U$
is a univalent function onto $U$ then 
\begin{equation}
\mathbf{E}(\tau^{\frac{p}{2}}_{U})<+\infty\Longleftrightarrow\Vert f\Vert_{\mathbf{H}^{p}(\mathbb{D})}<+\infty\Longleftrightarrow\mathbf{E}\left(\sup_{0\le s\le\tau_{U}}|Z_{s}|^{p}\right)<+\infty.\label{burkholder equivalence}
\end{equation}
In particular,
\begin{equation}
h(U)=2\sup\left\{ q>0:\ \mathbf{E}(\tau^{\,q}_{U})<\infty\right\} .\label{Brownian hardy number}
\end{equation}

As the right-hand side of \ref{Brownian hardy number} is meaningful

for any planar domain, we adopt it as the definition of the Brownian
Hardy number:
\[
h(U):=2\sup\left\{ q>0:\ \mathbf{E}(\tau^{q}_{U})<\infty\right\} .
\]

This extension was introduced in \cite{markowsky2015exit} under the
name Brownian Hardy number. The Hardy number encodes the heaviness
of the tail of $\tau_{U}$. If $U$ is bounded then obviously $h(U)=+\infty$.
If the domain $U$ is simply connected then its Hardy number can not
be too small. More precisely
\begin{thm}
\cite{duren2000theory} \label{thm:hardy lower bound} If $U$ is
a proper simply connected domain then $h(U)\geq\frac{1}{2}$. 
\end{thm}

In terms of Brownian exit times, Theorem \ref{thm:hardy lower bound}
says that for any simply connected domain $U$
\[
\mathbf{E}(\tau^{\,q}_{U})<+\infty
\]
whenever $q<\frac{1}{4}$. The information is handy to describe the
long time decay of the tail of $\tau_{U}$. By the layer cake representation
of the expectation, we have 
\[
\int^{+\infty}_{1}\frac{1}{t^{\gamma}}\mathbf{P}(\tau_{U}>t)dt<+\infty
\]
for any $\gamma>\frac{3}{4}$. In particular, no proper simply connected
domain $U$ can satisfy
\[
\mathbf{P}(\tau_{U}>t)\gtrsim t^{-\alpha}\qquad(t\to\infty)
\]
for some $\alpha<\frac{1}{4}$. In fact, the decay rate of $\frac{1}{\sqrt[4]{t}}$
in the long run is optimal and it is attained by the Koebe domain
$K=\mathbb{C}\setminus(-\infty,-\frac{1}{4}]$ \cite{brassesco1992,Betsakos2021a}. 

To avoid repeated qualifications, we adopt the following conventions
throughout the paper. All domains are open proper subsets of $\mathbb{C}$
and contain the origin, unless stated otherwise. Brownian motions
such as $(Z_{t})_{t\ge0}$ and $(B_{t})_{t\ge0}$ are standard planar
Brownian motions started at $0$. For a domain $U$, the notation
\[
\tau_{U}:=\inf\{t\ge0:\ Z_{t}\notin U\}
\]
always denotes the first exit time from $U$, with the underlying
Brownian motion clear from the context. The parameter $p$ is positive,
subject to any further assumptions stated locally. Finally, starlikeness
is always understood with respect to the origin.

The reader will notice that some proofs are given immediately after
the corresponding statements, while others are deferred to the proof
section. This organization of the proofs is intentional. In fact,
we keep a proof close to the statement whenever its argument contains
an idea that is conceptually useful for the development of the paper,
or when some part of the argument is used in subsequent results. Proofs
that are more technical, routine, or not needed for the immediate
flow of the exposition are postponed. A dagger symbol $\dagger$ is
added whenever the result proof is deferred.

\section{Some geometric aspects of the $p$-Brownian convergence }\label{sec:Some-aspects-of}

In this section, we illustrate some geometric properties of the $p$-Brownian
convergence. In particular, we show that $p$-Brownian convergence
provides a genuinely different way of tracking the asymptotic behavior
of a sequence of planar domains.
\begin{prop}
\label{prop:increasing-domains} Let $U$ be a proper simply connected
domain of $\mathbb{C}$ containing the origin, with positive $h(U)$,
and let $(U_{n})_{n\ge1}$ be a sequence of planar domains such that
\[
U_{1}\subset U_{2}\subset\cdots\subset U,\qquad\bigcup_{n\ge1}U_{n}=U.
\]
Then, for any $p<h(U)$, the sequence $(U_{n})_{n\geq1}$ converges
to $U$ in the $p$-Brownian sense. In fact, this convergence holds
under the same-path planar Brownian coupling. 
\end{prop}

\begin{proof}
Let $(Z_{t})_{t\ge0}$ be a planar Brownian motion started at $0$,
and define 
\[
\tau_{n}:=\inf\{t\ge0:\,Z_{t}\notin U_{n}\},\qquad\tau:=\inf\{t\ge0:\,Z_{t}\notin U\}.
\]
Since $U_{n}\subset U_{n+1}\subset U$, we have
\[
\tau_{n}\uparrow\tau\qquad\text{a.s.}
\]
By continuity of Brownian paths, it follows that 
\[
Z_{\tau_{n}}\longrightarrow Z_{\tau}\qquad\text{a.s.}
\]
Now fix $0<p<h(U)$. We have
\[
\mathbf{E}\left(\sup_{0\le s\le\tau}|Z_{s}|^{p}\right)<+\infty
\]
by \ref{burkholder equivalence}. In particular
\[
|Z_{\tau_{n}}-Z_{\tau}|^{p}\le2^{p}\sup_{0\le s\le\tau}|Z_{s}|^{p},
\]
Since $Z_{\tau_{n}}\to Z_{\tau}$ almost surely, the dominated convergence
theorem yields 
\[
\mathbf{E}\big(|Z_{\tau_{n}}-Z_{\tau}|^{p}\big)\underset{n\to+\infty}{\longrightarrow}0.
\]
Similarly, because $\tau_{n}\uparrow\tau$ almost surely, then again
the dominated convergence theorem yields 
\[
\mathbf{E}\big(|\tau_{n}-\tau|^{\frac{p}{2}}\big)\underset{n\to+\infty}{\longrightarrow}0.
\]
Thus 
\[
\mathbf{E}\left(\big|Z_{\tau_{n}}-Z_{\tau}\big|^{p}\right)+\mathbf{E}\left(\big|\tau_{n}-\tau\big|^{p/2}\right)\underset{n\to+\infty}{\longrightarrow}0.
\]
\end{proof}

The definition of $p$-Brownian convergence measures the proximity,
in a coupled $L^{p}$-sense, of the exit pairs $(\xi_{n},\theta_{n})$
and $(\xi,\theta)$. A natural question is whether $p$-Brownian convergence
follows from geometric closeness of the boundaries $\partial U_{n}$
and $\partial U$.
\begin{prop}
\label{prop:-There-exist} There exist bounded domains $U,(U_{n})_{n\geq1}\subset\mathbb{C}$
such that 
\[
\sup_{u\in\partial U}\text{dist}(u,\partial U_{n})=0\qquad\text{for all }n,
\]
but $(U_{n})_{n\geq1}$ does not converge to $U$ in the $p$-Brownian
sense for any $p>0$. 
\end{prop}

\begin{proof}
Let 
\[
U=D(0,2),\qquad U_{n}=D(0,2)\setminus[1,{\textstyle \frac{3}{2}}]\qquad\text{for all }n.
\]

\begin{center}
\begin{tikzpicture}[scale=0.7]
  \begin{scope}
    \draw (0,0) circle (2);
    \draw[->] (-2.4,0) -- (2.4,0) node[right] {$x$};
    \draw[->] (0,-2.4) -- (0,2.4) node[above] {$y$};
    \fill (0,0) circle (1pt);
    \node[below left] at (0,0) {$0$};
    \node at (0,-2.8) {$U=D(0,2)$};
  \end{scope}

  \begin{scope}[xshift=6cm]
    \draw (0,0) circle (2);
    \draw[->] (-2.4,0) -- (2.4,0) node[right] {$x$};
    \draw[->] (0,-2.4) -- (0,2.4) node[above] {$y$};
    \fill (0,0) circle (1pt);
    \node[below left] at (0,0) {$0$};
    \draw[very thick] (1,0) -- (1.5,0);
    \draw[fill=white] (1,0)   circle (1pt);
    \draw[fill=white] (1.5,0) circle (1pt);

    \node at (0,-2.8) {$U_n=D(0,2)\setminus[1,\tfrac{3}{2}]$};
  \end{scope}
\end{tikzpicture}
\end{center}
In particular 
\[
h(U)=h(U_{n})=+\infty\qquad\text{for all }n.
\]
Moreover, 
\[
\partial U\subset\partial U_{n},
\]
so 
\[
\sup_{u\in\partial U}\text{dist}(u,\partial U_{n})=0.
\]
We now and we show that $U_{n}$ does not converge to $U$ in the
$p$-Brownian sense. Let $(Z_{t})_{t}$ be planar Brownian motion
started at $0$, and let 
\[
\tau:=\tau_{U},\qquad\tau_{n}:=\tau_{U_{n}}.
\]
The slit $[1,\tfrac{3}{2}]$ is a non-polar portion of the boundary
of $U_{n}$. Hence its harmonic measure, seen from $0$, is positive:
\[
\delta:=\mathbf{P}\bigl(Z_{\tau_{n}}\in[1,\tfrac{3}{2}]\bigr)>0.
\]
On the other hand,
\[
Z_{\tau}\in\partial D(0,2)\qquad\text{a.s.}
\]
Every point of the slit $[1,\tfrac{3}{2}]$ is at Euclidean distance
at least $1/2$ from $\partial D(0,2)$. Therefore, for every coupling
of $Z_{\tau_{n}}$ and $Z_{\tau}$,
\[
|Z_{\tau_{n}}-Z_{\tau}|^{p}\ge\left(\frac{1}{2}\right)^{p}\quad\text{on the event }\{Z_{\tau_{n}}\in[1,\tfrac{3}{2}]\}.
\]
Consequently,

\[
\mathbf{E}\bigl(|Z_{\tau_{n}}-Z_{\tau}|^{p}\bigr)\ge\left(\frac{1}{2}\right)^{p}\delta>0\qquad\text{for all }n.
\]
Hence, $p$-Brownian convergence fails. 
\end{proof}

The Hausdorff distance between two subsets $A$ and $B$ of $\mathbb{C}$
(or, more generally, a metric space) is defined by 
\[
d_{H}(A,B)=\max(\sup_{a\in A}\text{dist}(a,B),\sup_{b\in B}\text{dist}(b,A)).
\]
The Hausdorff distance measures the worst possible distance from a
point of one set to the other set, in both directions. Proposition
\ref{prop:-There-exist} shows that one sided control of the distance
between $\partial U$ and $\partial U_{n}$ is not enough to imply
the $p$-Brownian convergence. The question about whether the condition
\begin{equation}
d_{H}(\partial U_{n},\partial U)\underset{n\to+\infty}{\longrightarrow}0\label{eq:d_H goes to zero}
\end{equation}
implies $p$-Brownian convergence is
a natural and nontrivial question. However, the next result shows
that \ref{eq:d_H goes to zero} and $p$-Brownian convergence are
not equivalent. 
\begin{thm}
\label{prop:brownian-not-hausdorff} For every $p>0$, there exists
a sequence of domains $(U_{n})_{n\ge1}$ converging to $U$ in the
$p$-Brownian sense while 
\[
d_{H}(\partial U_{n},\partial U)\not\to0.
\]
\end{thm}

\begin{proof}
Let 
\[
U:=\mathbb{D},
\]
and define 
\[
T_{n}:=\{re^{i\theta}:\ 1-\eta\leq r<3,\ |\theta|<\delta_{n}\},\qquad U_{n}:=\mathbb{D}\cup T_{n},
\]
where $\delta_{n}\downarrow0$ and $\eta$ sufficiently small. 

\begin{center}
\begin{tikzpicture}[scale=2]
  \def\R{1}        
  \def\Rout{2.2}   
  \def\Rin{0.82}   
  \def\dn{8}       

  \fill[gray!12] (0,0) circle (\R);
  \fill[gray!12]
    (\dn:\R) -- (\dn:\Rout)
    arc[start angle=\dn, end angle=-\dn, radius=\Rout]
    -- (-\dn:\R)
    arc[start angle=-\dn, end angle=\dn, radius=\R]
    -- cycle;
  \draw[->, gray] (-1.3,0) -- (2.5,0) node[below right, black] {$x$};
  \draw[->, gray] (0,-1.3) -- (0,1.3)  node[above left,  black] {$y$};

  \draw[thick] (0,0) circle (\R);
  \draw[thick] (\dn:\R) -- (\dn:\Rout)
               arc[start angle=\dn, end angle=-\dn, radius=\Rout]
               -- (-\dn:\R);

  \draw[densely dotted, thick] (\dn:\Rin) -- (\dn:\R);
  \draw[densely dotted, thick] (-\dn:\Rin) -- (-\dn:\R);
  \draw[densely dotted, thick]
       (\dn:\Rin) arc[start angle=\dn, end angle=-\dn, radius=\Rin];

  \node at (-0.35, 0.35) {$\mathbb{D}$};
  \node at ( 1.55, 0.35) {$T_n$};
\end{tikzpicture}
\end{center}Each $U_{n}$ is a bounded simply connected domain containing $0$.
We first note that Hausdorff convergence fails. Indeed, the point
$3$ belongs to $\partial U_{n}$ for every $n$, while 
\[
\text{dist}(3,\partial\mathbb{D})=2.
\]
Hence 
\[
d_{H}(\partial U_{n},\partial U)\ge2\qquad\text{for all }n,
\]
and therefore
\[
d_{H}(\partial U_{n},\partial U)\not\to0.
\]
We now prove that $U_{n}\to U$ in the $p$-Brownian sense using the
same path coupling. Let $(Z_{t})_{t\ge0}$ be a planar Brownian motion
started at $0$, and define 
\[
\tau:=\tau_{U}=\tau_{\mathbb{D}},\qquad\tau_{n}:=\tau_{U_{n}}.
\]
This gives a coupling of the exit pairs. Let 
\[
A_{n}:=\{e^{i\theta}:\ |\theta|<\delta_{n}\}\subset\partial\mathbb{D}.
\]
Since Brownian motion started at $0$ exits the unit disk according
to normalized arclength measure on $\partial\mathbb{D}$, we have
\begin{equation}
\mathbf{P}\bigl(Z_{\tau}\in A_{n}\bigr)=\frac{|A_{n}|}{2\pi}=\frac{2\delta_{n}}{2\pi}=\frac{\delta_{n}}{\pi}\underset{n\to+\infty}{\longrightarrow}0.\label{eq:first}
\end{equation}
On the event 
\[
E_{n}:=\{Z_{\tau}\notin A_{n}\},
\]
the Brownian path exits $\mathbb{D}$ at a boundary point which does
not belong to the attaching arc of the tentacle $T_{n}$. Since
\[
\partial\mathbb{D}\setminus A_{n}\subset\partial U_{n},
\]
it follows that 
\begin{equation}
\tau_{n}=\tau\qquad\text{and}\qquad Z_{\tau_{n}}=Z_{\tau}\qquad\text{on }E_{n}.\label{eq:second}
\end{equation}
We first estimate the exit-position term. By \ref{eq:second}, 
\[
|Z_{\tau_{n}}-Z_{\tau}|^{p}=0\qquad\text{on }E_{n}.
\]
Moreover, since $U_{n}\subset D(0,3)$ and $U=\mathbb{D}$, we have
\[
|Z_{\tau_{n}}|\le3,\qquad|Z_{\tau}|\le1,
\]
hence 
\[
|Z_{\tau_{n}}-Z_{\tau}|^{p}\le4^{p}\,\mathbf{1}_{E^{c}_{n}}.
\]
Taking expectations and using \ref{eq:first}, we obtain 
\begin{equation}
\mathbf{E}(|Z_{\tau_{n}}-Z_{\tau}|^{p})\le4^{p}\mathbf{P}(E^{c}_{n})\underset{n\to+\infty}{\longrightarrow}0\label{eq:third}
\end{equation}
We next estimate the exit-time term. Since $U\subset U_{n}$, we have
$\tau\le\tau_{n}$. Thus, on $E^{c}_{n}$, 
\[
|\tau_{n}-\tau|^{\frac{p}{2}}=(\tau_{n}-\tau)^{\frac{p}{2}}\le\tau^{\frac{p}{2}}_{n}.
\]
Because $U_{n}\subset D(0,3)$, 
\[
\tau_{n}\le\tau_{D(0,3)},
\]
and therefore 
\[
|\tau_{n}-\tau|^{\frac{p}{2}}\le\tau^{\,\frac{p}{2}}_{D(0,3)}\mathbf{1}_{E^{c}_{n}}\leq\tau^{\,\frac{p}{2}}_{D(0,3)}.
\]
The dominated convergence theorem yields 
\[
\mathbf{E}(|\tau_{n}-\tau|^{\frac{p}{2}})\to0.
\]
We see that under this coupling 
\[
\mathbf{E}(|Z_{\tau_{n}}-Z_{\tau}|^{p})+\mathbf{E}(|\tau_{n}-\tau|^{\frac{p}{2}})\to0.
\]
Hence $U_{n}\to U$ in the $p$-Brownian sense. 
\end{proof}

\section{A Brownian-Wasserstein distance }\label{sec:A-Brownian-Wasserstein-distance}

In the previous section, $p$-Brownian convergence is treated as a mode of convergence. Nevertheless, it is more convenient to exhibit a kind of metric out of it. That is, we pair such a convergence with a canonical metric between domains. The mechanism is by placing the Brownian convergence within the framework of transport theory. So first, we recall some concepts from the transport
theory.

Let $\mu_{1}$ and $\mu_{2}$ be probability measures on measurable
spaces $X$ and $Y$, respectively. A coupling of $\mu_{1}$ and $\mu_{2}$
is a pair of random variables $(\xi,\theta)$, defined on a common
probability space, such that
\[
\xi\sim\mu_{1},\qquad\theta\sim\mu_{2}.
\]
Equivalently, a coupling is a probability measure $\pi$ on $X\times Y$
whose marginals are $\mu_{1}$ and $\mu_{2}$. In transport theory, a coupling is also referred to as a transport
plan. The set of all such transport plans is denoted by $\Pi(\mu_{1},\mu_{2})$.
Given a cost function $c:X\times Y\to[0,\infty]$, the Kantorovich
transport problem is to minimize
\[
\int_{X\times Y}c(x,y)\,d\pi(x,y)
\]
over all $\pi\in\Pi(\mu_{1},\mu_{2})$. Equivalently, if $(\xi,\theta)$
has joint law $\pi$, the objective is to optimize the value of the expectation
\[
\mathbf{E}(c(\xi,\theta)).
\]
We refer the reader to \cite{ambrosio2021lectures,villani2021topics}
for a concise treatments of optimal transport theory.

As mentioned earlier, our transportation problem is motivated by the $p$-Brownian convergence.
Fix two domains $U$ and $V$, run a standard planar Brownian motion
$(Z_{t})_{t}$ and record the exit times from $U$ and $V$, i.e.
\[
\tau_{U}:=\inf\{t\mid Z_{t}\not\in U\},\,\,\,\tau_{V}:=\inf\{t\mid Z_{t}\not\in V\}.
\]
That is, a domain $U$ give rise to what we shall call its associated
``Brownian exit pair''
\[
(Z_{\tau_{U}},\tau_{U}).
\]
Denote by $\mu_{U}$ (resp. $\mu_{V}$) the joint law of $(Z_{\tau_{U}},\tau_{U})$
(resp. $(Z_{\tau_{V}},\tau_{V})$). Our set of transportation plans
is then 
\[
\Pi(\mu_{U},\mu_{V})
\]
and we will measure the cost of transporting one exit pair into the
other. In this language, different couplings correspond to different
ways of relating two Brownian exit pairs. The same-path coupling corresponds
to driving two Brownian particles by the same Brownian noise until they leave
their respective domains, while the independent coupling corresponds
to two particles evolving independently. The comparison between these
two couplings is one of the main themes of the paper. Therefore, the resulting
theory connects Brownian convergence of domains with ideas from optimal
transport. 
Let $E=\mathbb{C}\times[0,\infty)$ and $p\ge1$. Consider
the function $d_{p}$ on $E\times E$ defined by
\[
d_{p}\bigl((z,t),(w,s)\bigr):=\bigl(|z-w|^{p}+|t-s|^{\frac{p}{2}}\bigr)^{\frac{1}{p}}.
\]
The metric $d_{p}$ is the $\ell^{p}$-product version of the standard
parabolic space-time metric. The cases $p=1$ and $p=2$ are common
in parabolic PDE theory \cite{PolacikQuittnerSouplet2007}, while this general $p$-form
follows from the the $p$-product construction for metric spaces.
What is specific to the present work is its use as the ground metric
for transporting Brownian exit pairs. Our cost function, as suggested
by the $p$-Brownian convergence, is 
\[
c_{p}\bigl((z,t),(w,s)\bigr):=d_{p}\bigl((z,t),(w,s)\bigr)^{p}=\bigl(|z-w|^{p}+|t-s|^{\frac{p}{2}}\bigr).
\]
The subscript $p$ is omitted when the context is clear. Define the
transport cost of $\pi$ (transport in short) by 
\[
T^{\pi}_{p}=\left(\int_{\left(\mathbb{C}\times[0,\infty)\right)^{2}}c\bigl((z,t),(w,s)\bigr)\,d\pi\right)^{1/p}.
\]
By the elementary analysis techniques, there is a constant $C_{p}>0$
such that
\begin{equation}
c((z,t),(w,s))\le C_{p}\left(|z|^{p}+t^{p/2}+|w|^{p}+s^{p/2}\right).\label{eq:C_p}
\end{equation}
Therefore a uniform estimate for $T^{\pi}_{p}$ is 
\[
T^{\pi}_{p}\apprle\left(\mathbf{E}(|Z_{\tau_{U}}|^{p})+\mathbf{E}(\tau^{\frac{p}{2}}_{U})+\mathbf{E}(|Z_{\tau_{V}}|^{p})+\mathbf{E}(\tau^{\frac{p}{2}}_{V})\right)^{1/p}.
\]

\begin{defn}
Let $p\ge1$, and let $U,V$ be two planar domains containing $0$,
such that 
\begin{equation}
\mathbf{E}(\tau^{\frac{p}{2}}_{U})+\mathbf{E}(|Z_{\tau_{U}}|^{p})<+\infty\,\,\,\,\,\text{and}\,\,\,\,\mathbf{E}(|Z_{\tau_{V}}|^{p})+\mathbf{E}(\tau^{\frac{p}{2}}_{V})<+\infty.\label{eq:existing assumption}
\end{equation}
We define the optimal Brownian transport cost by
\[
\Lambda_{p}(U,V):=\inf_{\pi\in\Pi(\mu_{U},\mu_{V})}T^{\pi}_{p}.
\]
Similarly, we define the supremal Brownian transport cost by
\[
\Phi_{p}(U,V):=\sup_{\pi\in\Pi(\mu_{U},\mu_{V})}T^{\pi}_{p}.
\]
\end{defn}

The optimal Brownian transport cost is simply the $p$-Wasserstein
distance between the joint exit laws of $U$ and $V$. Since our space-time
framework $E$ is Polish and $$c=d^{p}_{p}$$ is continuous and nonnegative,
the usual compactness and lower-semicontinuity argument gives the
existence of an optimal coupling in $\Pi(\mu_{U},\mu_{V})$ \cite{villani2021topics}.
The existence of a supremal transport plan needs attention. Although
$c$ is continuous, it is not bounded. Define
\[
A(z,t):=C_{p}(|z|^{p}+t^{p/2}),\qquad B(w,s):=C_{p}(|w|^{p}+s^{p/2}),
\]
and
\begin{equation}
h((z,t),(w,s)):=A(z,t)+B(w,s)-c((z,t),(w,s))\label{eq:h for Phi}
\end{equation}
where $C_{p}$ is the constant in \ref{eq:C_p}. Then $h\ge0$, and
$h$ is lower semicontinuous. Moreover, for every $\pi\in\Pi(\mu_{U},\mu_{V})$,
\[
\int cd\pi=\int Ad\mu_{U}+\int Bd\mu_{V}-\int hd\pi.
\]
The first two terms are finite and independent of $\pi$. Hence maximizing
$\int c\,d\pi$ over $\Pi(\mu,\nu)$ is equivalent to minimizing $\int h\,d\pi$
over $\Pi(\mu,\nu)$. Since $E$ is a Polish space, the set $\Pi(\mu_{U},\mu_{V})$
is narrowly compact. Also, because $h$ is nonnegative and lower semicontinuous,
the map
\[
\pi\longmapsto\int h\,d\pi
\]
is lower semicontinuous. . We therefore reduce the maximization
problem to a standard minimization problem with a nonnegative lower
semicontinuous cost. The minimum of $\int h\,d\pi$
over $\Pi(\mu,\nu)$ is attained. Consequently the original supremum
is also attained.

The conditions on the moments growth under which optimal transport
exists deserve to be examined and discussed in some detail. In fact,
by Burkholder equivalence \ref{burkholder equivalence}, one always
has
\begin{equation}
\mathbf{E}(\tau^{\frac{p}{2}}_{U})<+\infty\Rightarrow\mathbf{E}(|Z_{\tau_{U}}|^{p})<+\infty.\label{time implies position}
\end{equation}
So we can drop the moment growth of the exit position. However, when
the exit time satisfies
\begin{equation}
\mathbf{E}(\log(1+\tau_{U}))<+\infty\label{burkholder condition}
\end{equation}
then the converse of \ref{time implies position} becomes true. Hence
it is enough to check only the moment growth of the exit position,
which is often easier to handle. Condition \ref{burkholder condition}
is fulfilled for simply connected domains \cite{boudabra2021some},
which constitute the main setting of the present paper. An situation
where the converse of \ref{time implies position} fails is, for example,
\[
U=\mathbb{C}\setminus\{|z-2|<1\}
\]
 (see \cite{boudabra2021some}). Recent work \cite{becher2026burkholder}
has also studied condition $\ref{burkholder condition}$ from the
viewpoint of analytic function theory. More precisely, if $U$ is
the range of a proper map $f:\mathbb{D}\to\mathbb{C}$ and satisfies
$\ref{burkholder condition}$, then $f$ belongs to the Smirnov
class, a functional space.\\
\\
A natural thought is that the temporal component might be controlled
by the spatial component through a Burkholder-{}Davis-{}Gundy type
inequality. This is not the case in the present transport framework.
BDG applies to increments of a single Brownian path between two stopping
times: if $\sigma\le\tau$, then the elapsed time $\tau-\sigma$ is
the quadratic variation of the Brownian increment on the interval
$[\sigma,\tau]$. The BDG says that 

\[
\mathbf{E}\left(\sup_{\sigma\le t\le\tau}|B_{t}-B_{\sigma}|^{p}\right)\asymp\mathbf{E}|\tau-\sigma|^{p/2}.
\]
In our setting, however, a coupling of exit pairs need not arise from
two points on the same Brownian trajectory. The variables
\[
(\xi,\theta)=(B_{\tau_{U}},\tau_{U}),\qquad(\xi',\theta')=(Z_{\tau_{V}},\tau_{V})
\]
are coupled abstractly through a transport plan, and there is in general
no Brownian increment connecting space to time. Hence there is no quadratic
variation identity relating $|\xi-\xi'|^{p}$ to $|\theta-\theta'|^{p/2}$.
This separation is already visible in the unit disc. For $\mathbb{D}$,
the exit position and exit time are independent \cite{port2012brownian}.
We may therefore couple two exit pairs by taking the same exit angle
but independent exit times:
\[
\xi=\xi',\,\,\,\,\,\,\theta'\independent\theta
\]
where $\theta'$ and $\theta$ are independent copies of $\tau_{\mathbb{D}}$.
Then both marginals are the correct Brownian exit-pair law of the
disc, and yet
\[
|\xi-\widetilde{\xi}|=0
\]
while
\[
|\theta-\theta'|
\]
is nontrivial. Thus the spatial component alone does not determine,
or even control, the temporal discrepancy under general couplings.
The time coordinate is therefore genuine additional information in
the Brownian exit-pair transport problem. This independent behavior
of the spatial and temporal components will appear again in Theorem
\ref{thm:same path vs ind path}.

In the definition of $\Lambda_{p}$, the pair of domains $U$ and
$V$ are not required to be simply connected. However, in order to
avoid pathological behavior of the boundaries, we assume always that
the inputs of $\Lambda_{p}$ belong to the set 
\[
\mathcal{S}_{p}=\{U:\text{proper simply connected domain of }\mathbb{C}\mid0\in U,\,\,\mathbf{E}(|Z_{\tau_{U}}|^{p})<+\infty\}.
\]
One also has
\[
\mathcal{S}_{p}=\{U:\text{proper simply connected domain of }\mathbb{C}\mid0\in U,\,\,\mathbf{E}(\tau^{p/2}_{U})<+\infty\}.
\]
It is clear then
\[
\mathcal{S}_{p+q}\subset\mathcal{S}_{p}
\]
for $q>0$, and for any $p<\frac{1}{2}$, 
\[
\mathcal{S}_{p}=\{U:\text{proper simply connected domain containing \ensuremath{0}}\}
\]
which is simply a reformulation of Theorem \ref{thm:hardy lower bound}.

In terms of measures, $\Lambda_{p}$ defines a distance, which is merely the Wasserstein distance. However,
since our measures are induced by domains, then it is natural to investigate
whether $\Lambda_{p}$ still defines a distance between domains or
not. 
\begin{thm}
\label{thm:Lambda is dist} Let $p\geq1$. $\Lambda_{p}$ defines
a distance on $\mathcal{S}_{p}$. Moreover, $\Lambda_{p}$ characterizes
the $p$-Brownian convergence, i.e.
\[
U_{n}\to U\text{ in the \ensuremath{p}-Brownian sense if and only if }\Lambda_{p}(U,U_{n})\to0.
\]
\end{thm}

\begin{proof}
We split the proof into two parts.
\begin{enumerate}
\item $\Lambda_{p}$ defines a distance on $\mathcal{S}_{p}$. Since $c$
is symmetric, for $U,V,W\in\mathcal{S}_{p}$, we immediately have
\[
\Lambda_{p}(U,V)\geq0,
\]
\[
\Lambda_{p}(U,V)=\Lambda_{p}(V,U),
\]
and 
\[
\Lambda_{p}(U,W)\leq\Lambda_{p}(U,V)+\Lambda_{p}(V,W).
\]
For the sake of completeness, we recall the proof of the triangle
inequality. Let $\varepsilon>0$. Choose couplings 
\[
\pi_{UV}\in\Pi(\mu_{U},\mu_{V}),\qquad\pi_{VW}\in\Pi(\mu_{V},\mu_{W})
\]
such that 
\[
\left(\int c(a,b)\,d\pi_{UV}(a,b)\right)^{\frac{1}{p}}\leq\Lambda_{p}(U,V)+\varepsilon,
\]
and 
\[
\left(\int c(b,c)\,d\pi_{VW}(b,c)\right)^{\frac{1}{p}}\leq\Lambda_{p}(V,W)+\varepsilon.
\]
By the so-called ``gluing lemma'', there exists a probability measure
$\gamma$ on $E^{3}$ whose $(a,b)$-marginal is $\pi_{UV}$ and whose
$(b,c)$-marginal is $\pi_{VW}$. The $(a,c)$-marginal of $\gamma$
is a coupling of $\mu_{U}$ and $\mu_{W}$. Therefore, using the triangle
inequality for $\rho$ along with Minkowski's inequality, we get 
\[
\begin{aligned}\Lambda_{p}(U,W) & \leq\left(\int_{E^{3}}c(a,c)\,d\gamma(a,b,c)\right)^{\frac{1}{p}}\\
 & \leq\left(\int_{E^{3}}\bigl(d_{p}(a,b)+d_{p}(b,c)\bigr)^{p}\,d\gamma(a,b,c)\right)^{\frac{1}{p}}\\
 & \leq\left(\int_{E^{3}}c(a,b)\,d\gamma(a,b,c)\right)^{\frac{1}{p}}+\left(\int_{E^{3}}c(b,c)\,d\gamma(a,b,c)\right)^{\frac{1}{p}}\\
 & =\left(\int_{E^{2}}c(a,b)\,d\pi_{UV}(a,b)\right)^{\frac{1}{p}}+\left(\int_{E^{2}}c(b,c)\,d\pi_{VW}(b,c)\right)^{\frac{1}{p}}\\
 & \leq\Lambda_{p}(U,V)+\Lambda_{p}(V,W)+2\varepsilon.
\end{aligned}
\]
Letting $\varepsilon\downarrow0$ gives the desired triangle inequality.
It remains to prove separation: if 
\[
\Lambda_{p}(U,V)=0,
\]
then $U=V$. If $\Lambda_{p}(U,V)=0$ then
\[
\mu_{U}=\mu_{V}.
\]
Taking spatial marginals gives 
\[
\mathcal{L}(B_{\tau_{U}})=\mathcal{L}(B_{\tau_{V}}).
\]
These spatial marginals are precisely the harmonic measures of $U$
and $\ensuremath{V}$seen from $0$ \cite{bass1994probabilistic,chung2012brownian}.
Thus 
\[
\omega^{0}_{U}=\omega^{0}_{V}.
\]
We now use the standard fact that if $D\subsetneq\mathbb{C}$ is simply
connected and $z_{0}\in D$, then 
\begin{equation}
\text{supp}(\omega^{z_{0}}_{D})=\partial D.\label{harmonic measure support}
\end{equation}
A nice picture-based proof of \ref{harmonic measure support} is in
\cite{2018remark}. Applying this to $U$ and $V$, we get $\partial U=\partial V$.
Since $U,V$ are simply connected with a common point (the origin)
then 
\[
U=V.
\]
Thus $\Lambda_{p}$ is nonnegative, symmetric, satisfies the triangle
inequality, and separates points of $\mathcal{S}_{p}$. Hence $\Lambda_{p}$
defines a distance on $\mathcal{S}_{p}$.
\item $\Lambda_{p}$ characterizes the $p$-Brownian convergence. The direct
implication is obvious. Now, suppose that 
\[
\Lambda_{p}(U,U_{n})\to0.
\]
Hence, there is always a coupling $\pi_{n}$ in $\Pi(\mu_{U},\mu_{U_{n}})$
such that 
\[
\int_{\left(\mathbb{C}\times[0,\infty)\right)^{2}}\rho\bigl((z,t),(w,s)\bigr)^{p}\,d\pi_{n}\leq\Lambda^{p}_{p}(U,U_{n})+\frac{1}{n}.
\]
taking $n$ to infinity proves the claim.
\end{enumerate}
\end{proof}

Note that the separation property of $\Lambda_{p}$ fails out of the
category of simply connected domains. To see this, observe that 
\[
\Lambda_{p}(\mathbb{D},\mathbb{D}\setminus\{{\textstyle \frac{1}{2}}\})=0.
\]

The next result connects the theory of Hardy spaces and the distance
$\Lambda_{p}$. 
\begin{thm}[Hardy-space control of the Brownian transport distance]
Let $p\ge1$, and let $U,V\in\mathcal{S}_{p}$ . Let 
\[
f:\mathbb{D}\to U,\qquad g:\mathbb{D}\to V
\]
be the normalized conformal maps satisfying
\[
f(0)=g(0)=0,\,\,f(\mathbb{D})=U,\,\,g(\mathbb{D})=V.
\]
Then there exists a constant $C_{p}>0$, depending only on $p$ (and
on the normalization convention for planar Brownian motion), such
that
\[
\Lambda_{p}(U,V)^{p}\le\|f-g\|^{p}_{\mathbf{H}^{p}}+C_{p}\|f-g\|^{\frac{p}{2}}_{\mathbf{H}^{p}}\|f+g\|^{\frac{p}{2}}_{\mathbf{H}^{p}}.
\]
\end{thm}

The proof is an adaptation of the canonical-coupling argument used
in the proof of the stability theorem for the planar Skorokhod embedding
problem in \cite{boudabra2026stability}, so we deemed to not include
it. 

As mentioned in section \ref{sec:A-Brownian-Wasserstein-distance},
special attention is given to the competition between the same-path
and independent couplings. In the case $p=2$, this competition admits
explicit formulas in terms of Brownian exit times, allowing us to
identify regimes where one coupling is cheaper than the other and
leading to the notion of a Brownian coupling threshold, a concept that sill be formalized in the last section. Let $(B_{t})_{t\ge0}$
and $(Z_{t})_{t\geq0}$ be two independent coupled planar Brownian
motions started at $0$. Define the same-path transport by
\[
T^{\mathrm{same}}_{p}(U,V):=\left(\mathbf{E}\left(|B_{\tau^{B}_{U}}-B_{\tau^{B}_{V}}|^{p}+|\tau^{B}_{U}-\tau^{B}_{V}|^{p/2}\right)\right)^{1/p}.,
\]
and define the independent-path cost by
\[
T^{\mathrm{ind}}_{p}(U,V):=\left(\mathbf{E}\left(|B_{\tau^{B}_{U}}-Z_{\tau^{Z}_{V}}|^{p}+|\tau^{B}_{U}-\tau^{Z}_{V}|^{p/2}\right)\right)^{1/p}.
\]
The question we address is whether one can compare the same-path and
independent-path costs via a geometric criterion. In general this problem
is not elementary, but for $p=2$ the comparison admits a simple expression
in terms of exit times. More precisely, we characterize the sign of

\[
T^{\mathrm{ind}}_{2}(U,V)-T^{\mathrm{same}}_{2}(U,V).
\]

\begin{thm}
\label{thm:same path vs ind path} Let $U,V\in\mathcal{S}_{2}$. Let
$(B_{t})_{t\ge0}$ and $(Z_{t})_{t\geq0}$ be two independent planar
Brownian motions started at $0$, and let $\theta$ and $S$ be independent
random variables such that
\[
\theta\sim\tau_{U},\qquad S\sim\tau_{V}.
\]
Then
\[
T^{\mathrm{ind}}_{2}(U,V)^{2}-T^{\mathrm{same}}_{2}(U,V)^{2}=2\left(3\,\mathbf{E}(\tau_{U\cap V})-\mathbf{E}(\theta\wedge S)\right).
\]
In particular,
\[
T^{\mathrm{same}}_{2}(U,V)<T^{\mathrm{ind}}_{2}(U,V)\quad\Longleftrightarrow\quad3\,\mathbf{E}(\tau_{U\cap V})>\mathbf{E}(\theta\wedge S).
\]
\end{thm}

\begin{proof}
We first compute the same-path cost. Set
\[
\tau:=\tau^{B}_{U},\qquad\sigma:=\tau^{B}_{V}.
\]
Since both stopping times are taken along the same Brownian path,
we have
\[
\tau\wedge\sigma=\tau_{U\cap V}.
\]
Also,
\[
|B_{\tau}-B_{\sigma}|=|B_{\tau\vee\sigma}-B_{\tau\wedge\sigma}|.
\]
By the Brownian isometry between stopping times,
\[
\mathbf{E}(|B_{\tau}-B_{\sigma}|^{2})=2\mathbf{E}(\tau\vee\sigma-\tau\wedge\sigma)=2\mathbf{E}(|\tau-\sigma|).
\]
Therefore
\[
T^{\mathrm{same}}_{2}(U,V)^{2}=3\,\mathbf{E}(|\tau-\sigma|).
\]
Since
\[
|\tau-\sigma|=\tau+\sigma-2(\tau\wedge\sigma),
\]
we obtain
\[
T^{\mathrm{same}}_{2}(U,V)^{2}=3\left(\mathbf{E}(\tau_{U})+\mathbf{E}(\tau_{V})-2\,\mathbf{E}(\tau_{U\cap V})\right).
\]
Now we compute the independent cost. Since $B$ and $Z$ are independent
and
\[
\mathbf{E}(B_{\tau^{B}_{U}})=\mathbf{E}(Z_{\tau^{Z}_{V}})=0,
\]
we have
\[
\mathbf{E}(|B_{\tau^{B}_{U}}-Z_{\tau^{Z}_{V}}|^{2})=\mathbf{E}(|B_{\tau^{B}_{U}}|^{2})+\mathbf{E}(|Z_{\tau^{Z}_{V}}|^{2}).
\]
By optional stopping applied to the martingales $|B_{t}|^{2}-2t$
and $|Z_{t}|^{2}-2t$,
\[
\mathbf{E}(|B_{\tau^{B}_{U}}|^{2})=2\mathbf{E}(\tau_{U}),\qquad\mathbf{E}(|Z_{\tau^{Z}_{V}}|^{2})=2\mathbf{E}(\tau_{V}).
\]
Hence
\[
\mathbf{E}(|B_{\tau^{B}_{U}}-Z_{\tau^{Z}_{V}}|^{2})=2(\mathbf{E}(\tau_{U})+\mathbf{E}(\tau_{V})).
\]
For the time term, since $\theta$ and $S$ are independent copies
of $\tau_{U}$ and $\tau_{V}$,
\[
\mathbf{E}(|\tau^{B}_{U}-\tau^{Z}_{V}|)=\mathbf{E}(|\theta-S|)=\mathbf{E}(\theta)+\mathbf{E}(S)-2\,\mathbf{E}(\theta\wedge S)=\mathbf{E}(\tau_{U})+\mathbf{E}(\tau_{V})-2\,\mathbf{E}(\theta\wedge S).
\]
Therefore
\[
T^{\mathrm{ind}}_{2}(U,V)^{2}=3\left(\mathbf{E}(\tau_{U})+\mathbf{E}(\tau_{V})\right)-2\,\mathbf{E}(\theta\wedge S).
\]
Subtracting the formula for $T^{\mathrm{same}}_{2}(U,V)$, we get
\[
\begin{aligned}T^{\mathrm{ind}}_{2}(U,V)^{2}-T^{\mathrm{same}}_{2}(U,V) & ^{2}=3\left(\mathbf{E}(\tau_{U})+\mathbf{E}(\tau_{V})\right)-2\,\mathbf{E}(\theta\wedge S)\\
 & \qquad-3\left(\mathbf{E}(\tau_{U})+\mathbf{E}(\tau_{V})-2\,\mathbf{E}(\tau_{U\cap V})\right)\\
 & =2\left(3\,\mathbf{E}(\tau_{U\cap V})-\mathbf{E}(\theta\wedge S)\right).
\end{aligned}
\]
This proves the identity, and the ``if and only if'' statement follows
immediately.
\end{proof}

For $p\neq2$, the spatial term no longer depends only on the elapsed
time between the two stopping times; it depends on the geometry of
the path during that interval. Thus, we speculate that no analogous criterion in terms
only of $\tau_{U},\tau_{V},\tau_{U\cap V}$ is available in general. 

Without the temporal component, the comparison between the same-path
and independent-path couplings becomes essentially tautological. Indeed,
if one uses only the spatial cost and $p=2$, then
\[
\left(T^{\mathrm{sp,same}}_{2}(U,V)\right)^{2}=\mathbf{E}\left(|B_{\tau_{U}}-B_{\tau_{V}}|^{2}\right).
\]
Since the two exit positions are taken from the same Brownian path,
the increment $B_{\tau_{U}}-B_{\tau_{V}}$ has quadratic variation
$|\tau_{U}-\tau_{V}|$, and hence
\[
\mathbf{E}\left(|B_{\tau_{U}}-B_{\tau_{V}}|^{2}\right)=2\mathbf{E}\left(|\tau_{U}-\tau_{V}|\right).
\]
Moreover,
\[
|\tau_{U}-\tau_{V}|=\tau_{U}+\tau_{V}-2(\tau_{U}\wedge\tau_{V})=\tau_{U}+\tau_{V}-2\tau_{U\cap V}.
\]
Therefore
\[
\left(T^{\mathrm{sp,same}}_{2}(U,V)\right)^{2}=2\left(\mathbf{E}\left(\tau_{U}\right)+\mathbf{E}\left(\tau_{V}\right)-2\mathbf{E}\left(\tau_{U\cap V}\right)\right).
\]
On the other hand, for independent Brownian paths,
\[
\left(T^{\mathrm{sp,ind}}_{2}(U,V)\right)^{2}=\mathbf{E}\left(|B_{\tau_{U}}-Z_{\tau_{V}}|^{2}\right)=2\mathbf{E}\left(\tau_{U}\right)+2\mathbf{E}\left(\tau_{V}\right).
\]
Consequently,
\[
\left(T^{\mathrm{sp,ind}}_{2}(U,V)\right)^{2}-\left(T^{\mathrm{sp,same}}_{2}(U,V)\right)^{2}=4\mathbf{E}(\tau_{U\cap V}).
\]
Since $0\in U\cap V$, the intersection contains a small ball around
the origin, and thus $\mathbf{E}(\tau_{U\cap V})>0$. Hence the same-path
coupling is automatically cheaper than the independent-path coupling
in the spatial-only comparison. In this sense, removing the time component
destroys the nontrivial balance present in the space-time cost: the
comparison no longer gives a meaningful criterion, but reduces to
the tautological fact that two Brownian paths which agree up to exiting
$U\cap V$ are spatially closer than two independent paths.
\begin{cor}
Let $p=2$. If $U\subset V$, then the same path-cost is cheaper than
the independent-path cost. 
\end{cor}

\begin{proof}
It follows from 
\[
3\mathbf{E}(\tau_{U\cap V})-\mathbf{E}(\theta\wedge S)\geq2\mathbf{E}(\tau_{U})>0.
\]
\end{proof}

Theorem \ref{thm:same path vs ind path} shows that the comparison
between $T^{\mathrm{ind}}_{2}$ and $T^{\mathrm{same}}_{2}$ is governed
by how much $U$ and $V$ are overlapping. It suggests that
\[
T^{\mathrm{ind}}_{2}(U,V)<T^{\mathrm{same}}_{2}(U,V)
\]
occurs when the overlap between $U$ and $V$ is small, and vice versa.
We give now a concrete illustration of this idea. Let $0<\varepsilon<L$,
and define
\[
U_{\varepsilon,L}:=\{z\in\mathbb{C}:\ -\varepsilon<\Re(z)<L\},\qquad V_{\varepsilon,L}:=\{z\in\mathbb{C}:\ -L<\Re(z)<\varepsilon\}.
\]

\begin{center}
\begin{tikzpicture}[scale=0.8]
  \begin{scope}
    \draw[->] (-2.7,0) -- (2.7,0) node[right] {$\Re (z)$};
    \draw[->] (0,-3.5) -- (0,3.5) node[above] {$\Im (z)$};

    \draw[dashed] (-0.3,-3.3) -- (-0.3, 3.3);
    \draw[dashed] ( 2.0,-3.3) -- ( 2.0, 3.3);
    \node[below] at (-0.3,-0.05) {$-\varepsilon$};
    \node[below] at ( 2.0,-0.05) {$L$};
    \node at (1.0, 2.3) {$U_{\varepsilon,L}$};
  \end{scope}
  \begin{scope}[xshift=8.5cm]
    \draw[->] (-2.7,0) -- (2.7,0) node[right] {$\Re (z)$};
    \draw[->] (0,-3.5) -- (0,3.5) node[above] {$\Im (z)$};

    \draw[dashed] ( 0.3,-3.3) -- ( 0.3, 3.3);
    \draw[dashed] (-2.0,-3.3) -- (-2.0, 3.3);
    \node[below] at ( 0.3,-0.05) {$\varepsilon$};
    \node[below] at (-2.0,-0.05) {$-L$};
    \node at (-1.0,-2.3) {$V_{\varepsilon,L}$};
  \end{scope}
\end{tikzpicture}
\end{center}

Both domains contain $0$, and neither is contained in the other.
Their intersection is the vertical strip
\[
U_{\varepsilon,L}\cap V_{\varepsilon,L}=\{z\in\mathbb{C}:\ |\Re(z)|<\varepsilon\}.
\]
We consider the comparison between the same-path and independent couplings
for $p=2$. Since the domains are invariant in the imaginary direction,
the exit time depends only on the one-dimensional Brownian motion
\[
X_{t}:=\Re(B_{t}).
\]
Thus
\[
\tau_{U_{\varepsilon,L}}=\inf\{t\geq0:\ X_{t}\notin(-\varepsilon,L)\},\qquad\tau_{V_{\varepsilon,L}}=\inf\{t\geq0:\ X_{t}\notin(-L,\varepsilon)\}.
\]
By symmetry, these two exit times have the same law. For a one-dimensional
Brownian motion started at $0$, the mean exit time from an interval
$(a,b)$ is
\[
\mathbf{E}(\tau_{(a,b)})=(0-a)(b-0)=-ab.
\]
Hence
\[
\mathbf{E}(\tau_{U_{\varepsilon,L}})=\mathbf{E}(\tau_{V_{\varepsilon,L}})=\varepsilon L.
\]
Moreover,
\[
U_{\varepsilon,L}\cap V_{\varepsilon,L}=\{|\Re(z)|<\varepsilon\},
\]
so
\[
\mathbf{E}(\tau_{U_{\varepsilon,L}\cap V_{\varepsilon,L}})=\varepsilon^{2}.
\]
Let
\[
\theta,\theta'
\]
be independent random variables with the common law of $\tau_{U_{\varepsilon,L}}$.
Then the comparison formula from the previous proposition yields
\[
T^{\mathrm{ind}}_{2}(U_{\varepsilon,L},V_{\varepsilon,L})^{2}-T^{\mathrm{same}}_{2}(U_{\varepsilon,L},V_{\varepsilon,L})^{2}=2\left(3\varepsilon^{2}-\mathbf{E}(\theta\wedge\theta')\right).
\]
Equivalently,
\[
T^{\mathrm{same}}_{2}(U_{\varepsilon,L},V_{\varepsilon,L})<T^{\mathrm{ind}}_{2}(U_{\varepsilon,L},V_{\varepsilon,L})\quad\Longleftrightarrow\quad\mathbf{E}(\theta\wedge\theta')<3\varepsilon^{2}.
\]
Since
\[
\mathbf{E}(|\tau_{U_{\varepsilon,L}}-\tau_{V_{\varepsilon,L}}|)=\mathbf{E}(\tau_{U_{\varepsilon,L}})+\mathbf{E}(\tau_{V_{\varepsilon,L}})-2\mathbf{E}(\tau_{U_{\varepsilon,L}\cap V_{\varepsilon,L}}),
\]
we obtain the explicit same-path cost
\[
T^{\mathrm{same}}_{2}(U_{\varepsilon,L},V_{\varepsilon,L})^{2}=3\Bigl(2\varepsilon L-2\varepsilon^{2}\Bigr)=6\varepsilon(L-\varepsilon).
\]
Similarly,
\[
T^{\mathrm{ind}}_{2}(U_{\varepsilon,L},V_{\varepsilon,L})^{2}=3\bigl(\mathbf{E}(\tau_{U_{\varepsilon,L}})+\mathbf{E}(\tau_{V_{\varepsilon,L}})\bigr)-2\mathbf{E}(\theta\wedge\theta')=6\varepsilon L-2\mathbf{E}(\theta\wedge\theta').
\]
This example shows that the sign of
\[
T^{\mathrm{ind}}_{2}-T^{\mathrm{same}}_{2}
\]
depends on the geometry of the overlap. For small $\frac{L}{\varepsilon}$,
the same-path coupling is cheaper; for sufficiently large $\frac{L}{\varepsilon}$,
the independent coupling becomes cheaper. A crude sufficient condition
for the same-path coupling to be cheaper is obtained from
\[
\mathbf{E}(\theta)=\varepsilon L<3\varepsilon^{2},
\]
equivalently when 
\[
\frac{L}{\varepsilon}<3.
\]

In the last section, the overlapping strips will be revisited and
discussed in a generalized setting.
\begin{prop}
\label{prop:monotonicity} The cost $T^{\text{ind}}_{2}$ is non decreasing
with respect to each argument.
\end{prop}

\begin{proof}
By symmetry, it is enough to show the property for one argument. Fix
$U$ and let $V\subset W$. By the formula established in the proof
of Theorem \ref{thm:same path vs ind path}, we have 
\[
T^{\text{ind}}_{2}(U,W)^{2}-T^{\text{ind}}_{2}(U,V)^{2}=3\left(\mathbf{E}\tau_{W}-\mathbf{E}\tau_{V}\right)-2(\mathbf{E}(\theta\wedge S_{W})-\mathbf{E}(\theta\wedge S_{V}))
\]
with $S_{V}\sim\tau_{V}$ , $S_{W}\sim\tau_{W}$, $\theta\sim\tau_{U}$,
$S_{V}\leq S_{W}$ and $S_{V},S_{W}$ are independent of $\theta$.
Since 
\[
\theta\wedge S_{W}-\theta\wedge S_{V}\leq S_{W}-S_{V}
\]
we get 
\[
T^{\text{ind}}_{2}(U,W)^{2}-T^{\text{ind}}_{2}(U,V)^{2}\geq\left(\mathbf{E}\tau_{W}-\mathbf{E}\tau_{V}\right)\geq0.
\]
\end{proof}

For the same-path coupling , one obtains only a weaker monotonicity
property .
\begin{prop}
If $U\subset V\subset W$ then 
\[
T^{\text{same}}_{2}(U,V)\leq T^{\text{same}}_{2}(U,W).
\]
\end{prop}

In other words, $T^{\text{same}}_{2}(U,\cdot)$ is not monotone under
$V\subset W$ unless $U$ is also nested with them. The proof is similar
to Proposition \ref{prop:monotonicity}. Nevertheless, it is worth
to point out the obstruction for $T^{\text{same}}_{2}(U,\cdot)$ to
be nondecreasing. For $T^{\text{same}}_{2}(U,\cdot)$ to be nondecreasing,
we need 
\[
|\tau_{U}-\tau_{V}|\leq|\tau_{U}-\tau_{W}|
\]
which is not necessary true as it depends where $\tau_{U}$ lies relative
to $\tau_{V}$ and $\tau_{W}$. 

\section{Scaled domains and explicit formulas }\label{sec:Scaled-domains-and}

Obtaining the exact value of $\Lambda_{p}$ and $\Phi_{p}$ seems
far from being obtained for an arbitrary choice of the inputs $U$
and $V$. However, when $U$ and $V$ are homothetic, then a lot can
be said as we will show in this section. In what comes next, as we
will work mainly with $U$ and its scaled copy 
\[
V:=\lambda U=\{\lambda z:\ z\in U\}.
\]
we will use the following quantities repeatedly:
\[
A_{p}:=\mathbf{E}(|Z_{\tau_{U}}|^{p}),\qquad M_{p}:=\mathbf{E}(\tau^{\frac{p}{2}}_{U}),\qquad S_{p}:=A_{p}+M_{p}.
\]

The first observation to state is that for large values of $\lambda$,
all costs between homothetic domains behave in the same. By Brownian
scaling, any coupling $\pi\in\Pi(\mu_{U},\mu_{\lambda U})$ may be
written as a coupling of two copies
\[
(\xi,\theta),\qquad(\xi',\theta')
\]
of the exit pair from $U$, with cost
\[
T^{\pi}_{p}=T^{\pi}_{p}(\lambda)=\left(\mathbf{E}_{\pi}(|\xi-\lambda\xi'|^{p}+|\theta-\lambda^{2}\theta'|^{p/2})\right)^{1/p}
\]

\begin{thm}
$\dagger$ \label{thm:expansion lamda} Let $p\ge2$, and let $U\subset\mathbb{C}$
be a bounded domain with $0\in U$. Then $T^{\pi}_{p}(\lambda)$ admits
a complete asymptotic expansion in powers of $\lambda^{-1}$. More
precisely, there exist constants
\[
b^{\pi}_{0},b^{\pi}_{1},b^{\pi}_{2},\ldots
\]
such that, for every $N\ge0$,
\begin{equation}
T^{\pi}_{p}(\lambda)=\lambda S^{1/p}_{p}+b^{\pi}_{0}+\sum^{N}_{j=1}b^{\pi}_{j}\lambda^{-j}+O(\lambda^{-N-1})\qquad(\lambda\to\infty).\label{eq:expansion}
\end{equation}
The first correction is
\[
b^{\pi}_{0}=-\frac{\mathbf{E}_{\pi}\left(|\xi'|^{p-2}\xi'\cdot\xi\right)}{S^{(p-1)/p}_{p}},
\]
where $\xi'\cdot\xi=\Re(\xi'\overline{\xi})$. In particular,
\[
T^{\pi}_{p}(\lambda)=\lambda S^{1/p}_{p}-\frac{\mathbf{E}_{\pi}\left(|\xi'|^{p-2}\xi'\cdot\xi\right)}{S^{(p-1)/p}_{p}}+O(\lambda^{-1}).
\]
\end{thm}

The boundedness of $U$ is used to obtain all positive spatial and
temporal moments automatically. The fact that $U$ contains a ball
around $0$ gives the lower-tail control needed for negative moments
of the exit time and for inverse powers of $|\xi'|$. For the first-order
expansion alone, boundedness can be dropped by the finite Brownian
$p$-moment assumption. In contrast, higher-order terms require the
corresponding higher positive moments.\\
\\
The coefficients $b^{\pi}_{j}$ involves moments of $\xi,\xi',\theta,\theta'$.
But the expansion reveals a useful large-scale hierarchy. The leading
term,
\[
\lambda S^{1/p}_{p},
\]
is independent of the coupling $\pi$; it only records the Brownian
size of the dilated domain. The first coupling-dependent correction
is the constant term
\[
-\frac{\mathbf{E}_{\pi}\left(|\xi'|^{p-2}\xi'\cdot\xi\right)}{S^{(p-1)/p}_{p}},
\]
which depends only on the coupling of the exit positions $\xi$ and
$\xi'$. The exit times do not appear at this order. Their first contribution
occurs only at order $\lambda^{-1}$. This reflects the parabolic
nature of the cost: spatial perturbations are of relative size $\lambda^{-1}$,
whereas temporal perturbations are of relative size $\lambda^{-2}$.
Consequently, in the large-dilation regime, the optimization suggests
an implicit hierarchical structure. 
One first optimizes the spatial coupling of the exit positions; the
temporal coupling is selected only at the next order. For the supremal
transport, the same hierarchy appears with the opposite spatial optimization.
When the coupling $\pi$ comes from two independent paths then the
constant term in \ref{eq:expansion} becomes zero and hence
\begin{equation}
T^{\mathrm{ind}}_{p}(\lambda)=\lambda S^{1/p}_{p}+O(\lambda^{-1}),\label{eq:inde asympto}
\end{equation}

\begin{thm}
$\dagger$ \label{thm:scaling optimal} Let $p\ge2$ and $U\in\mathcal{S}_{p}$
. Then for $\lambda>0$, 
\[
\Lambda_{p}(U,\lambda U)=\left(|\lambda-1|^{p}A_{p}+|\lambda^{2}-1|^{\frac{p}{2}}M_{p}\right)^{1/p}.
\]
\end{thm}

Theorem \ref{thm:scaling optimal} reads as follows: The diagonal
coupling 
\[
\pi=\mathcal{L}((\xi,\theta),(\lambda\xi,\lambda^{2}\theta))
\]
is the cheapest transport plan. Furthermore, we have 
\begin{equation}
\Lambda_{p}(U,\lambda U)=S_{p}{}^{1/p}\lambda-\dfrac{A_{p}}{S_{p}{}^{1-1/p}}+\dfrac{M_{p}\bigl((p-2)A_{p}-M_{p}\bigr)}{2S_{p}{}^{2-1/p}}\cdot\dfrac{1}{\lambda}+O(\lambda^{-2}).\label{eq:optimal asymp}
\end{equation}

The asymptotic estimates \ref{eq:inde asympto} and \ref{eq:optimal asymp}
show that the asymptotic absolute error between $T^{\mathrm{ind}}_{p}(\lambda)$
and $\Lambda_{p}(U,\lambda U)$ is
\[
\dfrac{A_{p}}{S_{p}{}^{1-1/p}}+O(\lambda^{-1}).
\]
In particular, the independent-path coupling cannot be the optimal
transport plan for large $\lambda$. When $p=2$, the optional stopping
theorem applied to the martingale $(|Z_{t}|^{2}-2t)_{t}$ yields the
relation
\[
M_{2}=\frac{1}{2}A_{2}.
\]

\begin{cor}
For $p=2$, $\lambda\geq1$, we have 
\[
\Lambda_{2}(U,\lambda U)=\sqrt{(3\lambda^{2}-4\lambda+1)M_{2}}.
\]
In particular 
\[
\Lambda_{2}(D(0,r),D(0,R))=\sqrt{\frac{\bigl(2(R-r)^{2}+|R^{2}-r^{2}|\bigr)}{2}}.
\]
 
\end{cor}

Unlike the transport $\Lambda_{p}$, the supremal Brownian coupling
$\Phi_{p}$ has no explicit formula for scaled domains. However, we
provide an integral formula for scaled discs. This case is special
because the exit pair has independent components \cite{port2012brownian}. 
\begin{prop}
$\dagger$ \label{prop:Unit disc case} Let $p\geq2$ and $\lambda>0$.
Let $Q$ be the quantile function of $\tau_{\mathbb{D}}$. Then
\[
\Phi_{p}(\mathbb{D},\lambda\mathbb{D})^{p}=(1+\lambda)^{p}+\int^{1}_{0}\left|Q(r)-\lambda^{2}Q(1-r)\right|^{p/2}\,dr.
\]
On the other hand, the independent coupling gives
\[
T^{\mathrm{ind}}_{p}(\mathbb{D},\lambda\mathbb{D})^{p}=I_{p}(\lambda)+J_{p/2}(\lambda),
\]
where
\[
I_{p}(\lambda):=\frac{1}{2\pi}\int^{2\pi}_{0}\left(1+\lambda^{2}-2\lambda\cos\theta\right)^{p/2}\,d\theta,
\]
and
\[
J_{p/2}(\lambda):=\int^{1}_{0}\int^{1}_{0}\left|Q(r)-\lambda^{2}Q(u)\right|^{p/2}\,dr\,du.
\]
\end{prop}

\begin{prop}
$\dagger$ \label{prop:supremal and independent disc}Let $p\geq2$
and $\lambda\geq1$. 
\[
\Phi_{p}(\mathbb{D},\lambda\mathbb{D})-T^{\mathrm{ind}}_{p}(\mathbb{D},\lambda\mathbb{D})\ge\frac{(1+\lambda)^{p}-(\lambda-1)^{p}}{2p(1+\lambda)^{p-1}\left(1+\mathbf{E}(\tau^{p/2}_{\mathbb{D}})\right)^{(p-1)/p}}.
\]
In particular,
\[
\Phi_{p}(\mathbb{D},\lambda\mathbb{D})-T^{\mathrm{ind}}_{p}(\mathbb{D},\lambda\mathbb{D})\ge\frac{1}{p\left(1+\mathbf{E}(\tau^{p/2}_{\mathbb{D}})\right)^{(p-1)/p}}.
\]
\end{prop}

Proposition \ref{prop:supremal and independent disc} implies that
the independent-path coupling is not supermal, neither for a fixed
$\lambda$ nor asymptotically, in this special case.

By applying the established formulas of $T^{\mathrm{same}}_{2}$ and
$T^{\mathrm{ind}}_{2}$ in the proof of Theorem \ref{thm:same path vs ind path}
to the pair $(U,\lambda U)$ when $U$ is starlike and $\lambda\geq1$,
we get 
\[
T^{\mathrm{same}}_{2}(U,\lambda U)=\sqrt{3(\lambda^{2}-1)M_{2}}
\]
\[
T^{\mathrm{ind}}_{2}(U,\lambda U)=\sqrt{3(1+\lambda^{2})M_{2}-2\mathbf{E}(\theta\wedge\lambda^{2}\theta')}
\]
where $\theta,\theta'$ are two independent copies of $\tau_{U}$.
In particular, we obtain
\begin{equation}
T^{\mathrm{same}}_{2}(U,\lambda U)\leq\sqrt{(3\lambda^{2}+1)M_{2}}\leq T^{\mathrm{ind}}_{2}(U,\lambda U)\leq\sqrt{3(\lambda^{2}+1)M_{2}}.\label{eq:ind same inequality}
\end{equation}
In particular,
\[
T^{\mathrm{ind}}_{2}(U,\lambda U)\ge2\sqrt{M_{2}},\qquad\lambda\ge1.
\]
Thus the independent-path cost admits a positive lower bound independent
of $\lambda$. Moreover, this lower bound is not attained: there is
no such domain $U$ for which
\[
T^{\mathrm{ind}}_{2}(U,U)=2\sqrt{M_{2}}.
\]

\begin{figure}[H]
\centering{}\includegraphics[width=10cm,totalheight=10cm,keepaspectratio]{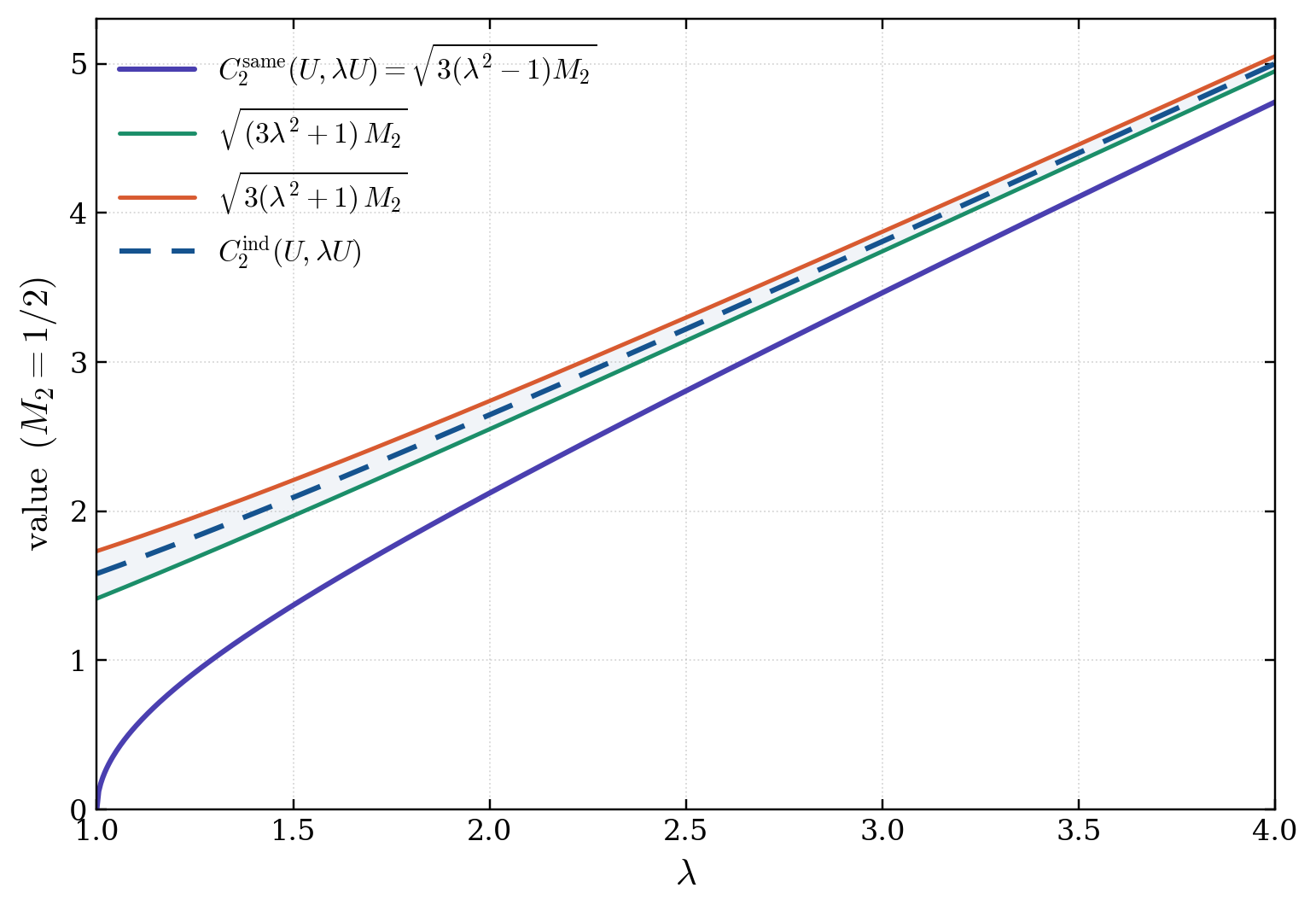}\caption{Plots of the costs for $p=2$. $M_{2}$ is set $\frac{1}{2}$, which
corresponds to $U=\mathbb{D}$ for example.}
\end{figure}

\begin{cor}
\label{thm:same pth dilation} Let $U\in\mathcal{S}_{2}$ be a bounded
domain which is starlike with respect to the origin. Then
\[
T^{\mathrm{same}}_{2}(U,\lambda U)-\Lambda_{2}(U,\lambda U)\geq\sqrt{\frac{2(\lambda-1)}{3\lambda+1}M_{2}}.
\]
In particular, when $\lambda>1$, neither the same-path nor the independent-path
couplings are optimal. 
\end{cor}

\begin{prop}
Let $U\subset V$ be bounded domains. Then 
\[
\sqrt{(3(1+\frac{\delta}{R})^{2}+1)M_2}\leq T^{\mathrm{ind}}_{2}(U,V),
\]
where $R:=\sup_{z\in\overline{U}}|z|$ and $\delta:=\text{dist}(\overline{U},V^{c})$.
\end{prop}

\begin{proof}
For $\lambda=1+\frac{\delta}{R}$ we have $\lambda U\subset\overline{V}$.
By the monotonicity property in Proposition \ref{prop:monotonicity}
along with \ref{eq:ind same inequality}, we get 
\[
\sqrt{(3(1+\frac{\delta}{R})^{2}+1)\mathbf{E}(\tau_{U})}\leq T^{\mathrm{ind}}_{2}(U,\lambda U)\leq T^{\mathrm{ind}}_{2}(U,V).
\]
\end{proof}

\section{Lifting with Brownian motion }\label{sec:Lifting-with-Brownian}

The definition of the optimal and supremal Brownian costs is made
at the level of exit-pair laws. In particular, neither the optimal
nor the supremal costs need to be the same-path or the independent-path
couplings. Nevertheless, the connection with certain Brownian paths
is still obtainable. Let
\[
\Omega:=\mathcal{C}([0,\infty),\mathbb{C})
\]
be the space of continuous paths in the plane. Let $\mathbf{W}$denote
planar Wiener measure on $\Omega$, started at the origin. For a domain
$D$ containing $0$, define
\[
\tau_{D}(\omega):=\inf\{t\geq0:\omega(t)\notin D\},
\]
and define the exit-pair map
\[
\Xi_{D}:\Omega\to\mathbb{C}\times[0,\infty),\qquad\Xi_{D}(\omega)=\bigl(\omega(\tau_{D}(\omega)),\tau_{D}(\omega)\bigr).
\]
Thus
\[
\mu_{D}=(\Xi_{D})_{\#}\mathbf{W}
\]
is the Brownian exit-pair law associated with $D$.
\begin{thm}
\label{thm:=00005BLifting-of-exit-pair}Let $U,V\in\mathcal{S}_{p}$,
and let
\[
\pi\in\Pi(\mu_{U},\mu_{V})
\]
be any coupling of the two exit-pair laws. Then there exists a coupling
\[
\Gamma\in\Pi(\mathbf{W},\mathbf{W})
\]
of two Brownian path laws such that
\[
(\Xi_{U},\Xi_{V})_{\#}\Gamma=\pi.
\]
Equivalently, there is a probability space supporting two Brownian
paths $(B_{t})_{t\geq0}$ and $(Z_{t})_{t\geq0}$, both started at
$0$, such that
\[
\bigl((B_{\tau^{B}_{U}},\tau^{B}_{U}),(Z_{\tau^{Z}_{V}},\tau^{Z}_{V})\bigr)\sim\pi,
\]
where
\[
\tau^{B}_{U}:=\inf\{t\geq0:B_{t}\notin U\},\qquad\tau^{Z}_{V}:=\inf\{t\geq0:Z_{t}\notin V\}.
\]
\end{thm}

\begin{proof}
Since $\Omega$ is a Polish space, we may disintegrate Wiener measure
with respect to the exit-pair maps. Thus there exist regular conditional
probability kernels $\vartheta_{U}$ and $\vartheta_{V}$ such that
\[
\mathbf{W}(d\omega)=\int\vartheta_{U}(a,d\omega)\,\mu_{U}(da),
\]
and
\[
\mathbf{W}(d\eta)=\int\vartheta_{V}(b,d\eta)\,\mu_{V}(db).
\]
Here $\vartheta_{U}(a,\cdot)$ is the conditional law of a Brownian
path given $\Xi_{U}=a$, and similarly for $\vartheta_{V}$. Given
$\pi\in\Pi(\mu_{U},\mu_{V})$, define a probability measure $\Gamma$
on $\Omega\times\Omega$ by
\[
\Gamma(d\omega,d\eta):=\int\vartheta_{U}(a,d\omega)\,\vartheta_{V}(b,d\eta)\,\pi(da,db).
\]
Its first marginal is
\[
\int\vartheta_{U}(a,d\omega)\,\mu_{U}(da)=\mathbf{W}(d\omega),
\]
and its second marginal is
\[
\int\vartheta_{V}(b,d\eta)\,\mu_{V}(db)=\mathbf{W}(d\eta).
\]
Hence $\Gamma\in\Pi(\mathbf{W},\mathbf{W})$. Moreover, since $\vartheta_{U}(a,\cdot)$-almost
surely $\Xi_{U}=a$, and $\vartheta_{V}(b,\cdot)$-almost surely $\Xi_{V}=b$,
we obtain
\[
(\Xi_{U},\Xi_{V})_{\#}\Gamma=\pi.
\]
This proves the claim.
\end{proof}

The lifting Theorem \ref{thm:=00005BLifting-of-exit-pair} may be
interpreted as follows. The transport plan $\pi$ specifies how the
exit pairs from $U$ are coupled with the exit pairs from $V$. It
does not, by itself, specify how the full Brownian paths should be
coupled. The disintegration of Wiener measure fills in this missing
path information: after sampling a pair of exit data $(a,b)$ according
to $\pi$, one samples independently a Brownian path conditioned on
$\Xi_{U}=a$ and a Brownian path conditioned on $\Xi_{V}=b$. The
resulting path-level coupling has Wiener measure as both marginals,
and its exit-pair pushforward is precisely $\pi$.
\begin{cor}
Let $U,V\in\mathcal{S}_{p}$. Then there exist two Brownian paths
$(B_{t})_{t\geq0}$ and $(Z_{t})_{t\geq0}$, both started at $0$,
such that
\[
\Lambda_{p}(U,V)^{p}=\mathbf{E}\left(|B_{\tau^{B}_{U}}-Z_{\tau^{Z}_{V}}|^{p}+|\tau^{B}_{U}-\tau^{Z}_{V}|^{\frac{p}{2}}\right).
\]
Similarly, there exist two Brownian paths $(\widetilde{B}_{t})_{t\geq0}$
and $(\widetilde{Z}_{t})_{t\geq0}$, both started at $0$, such that
\[
\Phi_{p}(U,V)^{p}=\mathbf{E}\left(|\widetilde{B}_{\tau^{\widetilde{B}}_{U}}-\widetilde{Z}_{\tau^{\widetilde{Z}}_{V}}|^{p}+|\tau^{\widetilde{B}}_{U}-\tau^{\widetilde{Z}}_{V}|^{\frac{p}{2}}\right).
\]
\end{cor}

The previous corollary should not be confused with the same-path coupling.
The coupling $\Gamma^{\ast}$ need not be supported on pairs $(\omega,\omega)$.
In general, the two paths are not equal, and the construction is not
necessarily adapted to a common Brownian filtration. The paths are
Brownian in their marginal laws, and their dependence is created through
the optimally coupled exit pairs.

\section{Numerical framework }\label{sec:Numerical-framework}

Now we provide a numerical framework in which we illustrate how to
construct the two Brownian paths. The scheme is based on the matching
problem. The finite matching problem asks for an optimal one-to-one
assignment between two finite sets. Given two sets
\[
A=\{a_{1},\ldots,a_{n}\},\qquad B=\{b_{1},\ldots,b_{n}\},
\]
and a cost matrix $C=(c_{ij})$, one seeks a permutation $\sigma\in S_{n}$
minimizing
\[
\sum^{n}_{i=1}c_{i,\sigma(i)}.
\]
Equivalently, this is the problem of finding a minimum-weight perfect
matching in a complete bipartite graph. Although a brute-force search
over all permutations has factorial complexity, the problem is polynomial-time
solvable. The Hungarian algorithm, introduced by Kuhn and refined
by Munkres, gives a classical polynomial-time method; standard implementations
solve the square assignment problem in $O(n^{3})$ time. Thus the
matching problem is combinatorial but computationally tractable in
its classical linear form, in contrast with higher-order assignment
variants, which are typically much harder. We refer the reader to
\cite{peyre2019computational} for more details about the algorithmic
optimal transport theory. For empirical measures with equal masses,
\[
\mu=\frac{1}{n}\sum^{n}_{i=1}\delta_{x_{i}},\qquad\nu=\frac{1}{n}\sum^{n}_{j=1}\delta_{y_{j}}.
\]
The Kantorovich optimal transport problem reduces to the linear assignment,
or bipartite matching, problem:
\[
\min_{\sigma\in S_{n}}\frac{1}{n}\sum^{n}_{i=1}c(x_{i},y_{\sigma(i)}).
\]
Equivalently, the optimal transport matrix may be chosen to be a permutation
matrix. 

We adapt this matching scheme to our problem as follows. Suppose that
we simulate $N$ independent standard Brownian paths in $U$,
\[
B^{1},\ldots,B^{N},
\]
and $N$ independent Brownian paths in $V$,
\[
Z^{1},\ldots,Z^{N}.
\]
Let
\[
a_{i}:=\bigl(B^{i}_{\tau^{i}_{U}},\tau^{i}_{U}\bigr),\qquad b_{j}:=\bigl(Z^{j}_{\tau^{j}_{V}},\tau^{j}_{V}\bigr)
\]
be their underlying exit pairs. The empirical exit-pair laws are
\[
\widehat{\mu}^{N}_{U}=\frac{1}{N}\sum^{N}_{i=1}\delta_{a_{i}},\qquad\widehat{\mu}^{N}_{V}=\frac{1}{N}\sum^{N}_{j=1}\delta_{b_{j}}.
\]
The empirical optimal transport problem is then the assignment problem
\[
\sigma\in S_{N}\longmapsto\frac{1}{N}\sum^{N}_{i=1}d_{p}(a_{i},b_{\sigma(i)})^{p},
\]
Let $\sigma_{N}$ (resp. $\varpi_{N}$) be a minimizing permutation
(resp. a maximizing permutation) . Then
\[
\widehat{\pi}_{\sigma_{N}}:=\frac{1}{N}\sum^{N}_{i=1}\delta_{(a_{i},b_{\sigma_{N}(i)})}
\]
is an optimal coupling of the empirical exit-pair laws $\ensuremath{\widehat{\mu}^{N}_{U}}\text{ and }\ensuremath{\widehat{\mu}^{N}_{V}}$.
Similarly 
\[
\widehat{\pi}_{\varpi_{N}}:=\frac{1}{N}\sum^{N}_{i=1}\delta_{(a_{i},b_{\varpi_{N}(i)})}
\]
is a supremal coupling of the empirical exit-pair laws $\ensuremath{\widehat{\mu}^{N}_{U}}\text{ and }\ensuremath{\widehat{\mu}^{N}_{V}}$.
The corresponding empirical path coupling is
\[
\widehat{\Gamma}_{\sigma_{N}}:=\frac{1}{N}\sum^{N}_{i=1}\delta_{(B^{i},Z^{\sigma_{N}(i)})}
\]
for the optimal coupling and 
\[
\widehat{\Gamma}_{\varpi_{N}}:=\frac{1}{N}\sum^{N}_{i=1}\delta_{(B^{i},Z^{\varpi_{N}(i)})}
\]
for the supremal one. By construction,
\[
(\Xi_{U},\Xi_{V})_{\#}\widehat{\Gamma}_{\sigma_{N}}=\widehat{\pi}_{\sigma_{N}},\,\,\,\,(\Xi_{U},\Xi_{V})_{\#}\widehat{\Gamma}_{\varpi_{N}}=\widehat{\pi}_{\varpi_{N}}.
\]
Thus the numerical procedure consists of two steps: first solve the
finite optimal transport problem for the exit pairs, and then pair
the entire Brownian paths according to the optimal matching. The empirical
optimal value is
\[
\widehat{\Lambda}_{p,N}(U,V)^{p}=\frac{1}{N}\sum^{N}_{i=1}\left(|B^{i}_{\tau^{i}_{U}}-Z^{\sigma_{N}(i)}_{\tau^{\sigma_{N}(i)}_{V}}|^{p}+|\tau^{i}_{U}-\tau^{\sigma_{N}(i)}_{V}|^{\frac{p}{2}}\right)=\mathbf{E}_{\widehat{\Gamma}_{\sigma_{N}}}\left[d_{p}\bigl(\Xi_{U}(\omega),\Xi_{V}(\eta)\bigr)^{p}\right].
\]
Similarly, the empirical supremal value  is 
\[
\widehat{\Phi}_{p,N}(U,V)^{p}=\frac{1}{N}\sum^{N}_{i=1}\left(|B^{i}_{\tau^{i}_{U}}-Z^{\varpi_{N}(i)}_{\tau^{\varpi_{N}(i)}_{V}}|^{p}+|\tau^{i}_{U}-\tau^{\varpi_{N}(i)}_{V}|^{\frac{p}{2}}\right)=\mathbf{E}_{\widehat{\Gamma}_{\varpi_{N}}}\left[d_{p}\bigl(\Xi_{U}(\omega),\Xi_{V}(\eta)\bigr)^{p}\right].
\]
Note that this empirical construction does not select one single best
pair of paths. It produces a global matching of the whole simulated
$U$-sample with the whole simulated $V$-sample. A randomly chosen
matched pair
\[
(B^{i},Z^{\sigma_{N}(i)})
\]
for example, is then one sample from the empirical path coupling $\widehat{\Gamma}_{\sigma_{N}}$.
Furthermore, if the exit pairs are sampled as described above and
$U,V\in\mathcal{S}_{p}$. Then, almost surely,
\[
\widehat{\Lambda}_{p,N}(U,V)\longrightarrow\Lambda_{p}(U,V)\qquad\text{as }N\to+\infty
\]
and 
\[
\widehat{\Phi}_{p,N}(U,V)\longrightarrow\Phi_{p}(U,V)\qquad\text{as }N\to+\infty.
\]
Note that the convergence of $\widehat{\Phi}_{p,N}(U,V)$ is not as
immediate as the convergence of $\widehat{\Lambda}_{p,N}(U,V)$. However,
it follows from the positivity of $h$ appeared in \ref{eq:h for Phi}. 

\begin{figure}
\begin{centering}
\includegraphics[width=15cm,totalheight=15cm,keepaspectratio]{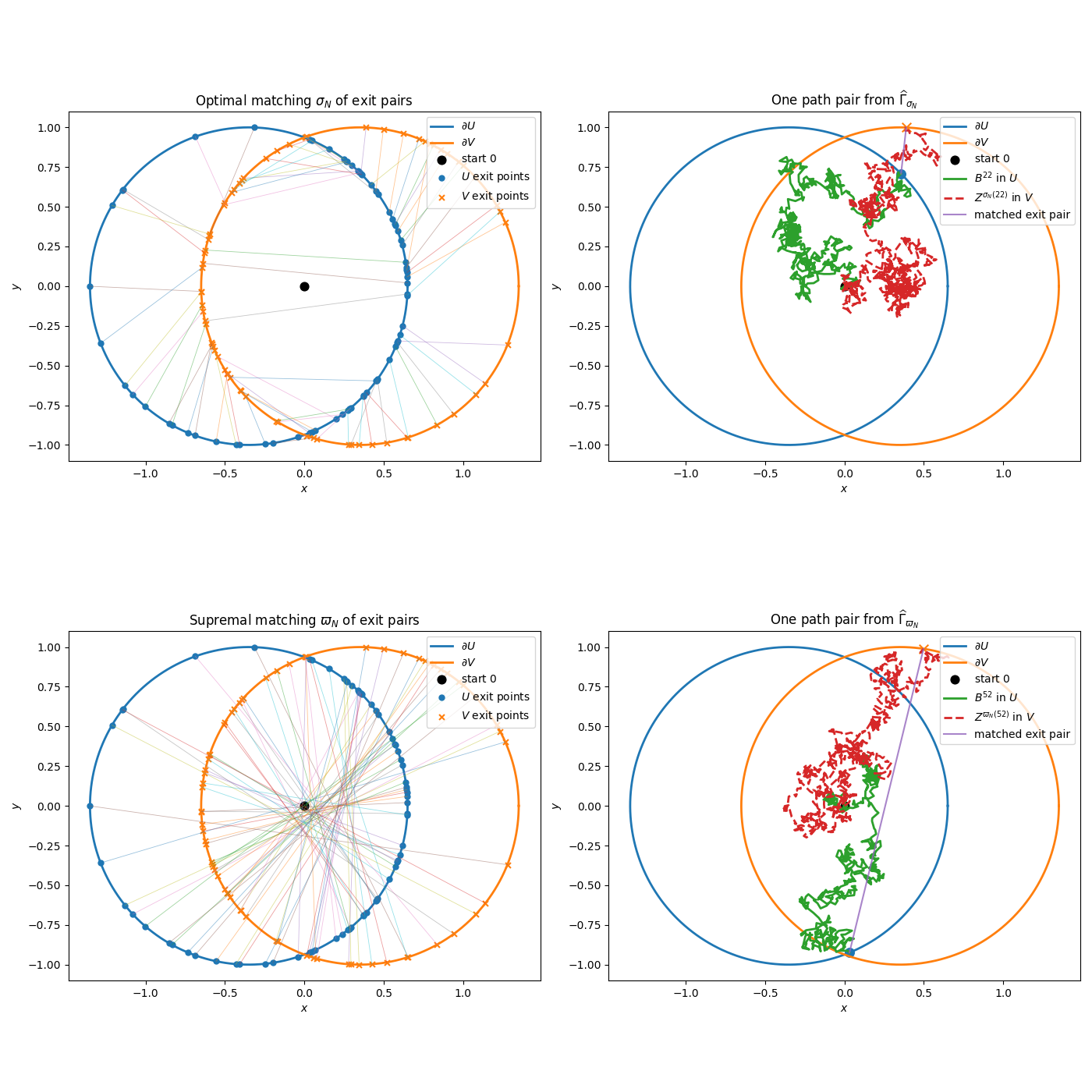}\caption{Empirical Brownian transport between two non-nested domains $U$ and
$V$. }
\par\end{centering}
The left panels show the matching of simulated Brownian exit pairs
\[
a_{i}=(B^{i}_{\tau^{i}_{U}},\tau^{i}_{U}),\,\,\,\,b_{j}=(Z^{j}_{\tau^{j}_{V}},\tau^{j}_{V}),
\]
using the Brownian-{}-parabolic cost
\[
d_{p}(a_{i},b_{j})^{p}=|B^{i}_{\tau^{i}_{U}}-Z^{j}_{\tau^{j}_{V}}|^{p}+|\tau^{i}_{U}-\tau^{j}_{V}|^{p/2}.
\]
The top-left panel corresponds to the minimizing permutation $\sigma_{N}$,
while the bottom-left panel corresponds to the maximizing permutation
$\varpi_{N}$. The right panels display one atom of the corresponding
empirical path coupling: the top-right panel shows a matched pair
$(B^{i},Z^{\sigma_{N}(i)})$ sampled from $\widehat{\Gamma}_{\sigma_{N}}$,
and the bottom-right panel shows a matched pair $(B^{i},Z^{\varpi_{N}(i)})$
sampled from$\widehat{\Gamma}_{\varpi_{N}}$. Thus the matching is
performed at the exit-pair level, while the resulting pairing is lifted
to the full simulated Brownian paths.
\end{figure}

The lifting mechanism also gives a concrete physical interpretation
of the transport plans considered above. One may regard the two Brownian
paths as the trajectories of two rapidly diffusing particles, started
from the same point and observed until they leave the prescribed domains.
A coupling of exit pairs first prescribes how the two final space-time
observations are matched; the lifting theorem then realizes this matching
by a coupling of the full particle trajectories. In this interpretation,
the same-path coupling corresponds to forcing the two particles to
follow the same random trajectory until the relevant exit times, while
the independent-path coupling corresponds to independent particle
trajectories. The preceding results show that neither of these two
natural mechanisms is universally optimal or universally extremal.
The least or largest Brownian transport cost may require a more subtle
organization of the paths, determined by the optimal coupling of their
exit-pair laws rather than by an a priori choice of identical or independent
motion.

\section{Remaining Proofs}\label{sec:Proof}
\begin{proof}
(Theorem \ref{thm:expansion lamda}) Since $U$ is a bounded domain
containing $0$, there exist constants $0<r<R<\infty$ such that
\[
r\le|\xi|,|\xi'|\le R\qquad\pi\text{-a.s.}
\]
Moreover,
\[
\tau_{D(0,r)}\le\theta,\theta'\le\tau_{D(0,R)}.
\]
The exit time of a disc has finite positive moments of every order,
and its lower tail is exponentially small near $0$. Consequently
$\theta$ and $\theta'$ have finite positive and negative moments
of every order. Set
\[
q:=\frac{p}{2},\qquad B_{p}:=S^{1/p}_{p}.
\]
We first expand the $p$-th power of $T^{\pi}_{p}(\lambda)$. For
the spatial term, write
\[
|\xi-\lambda\xi'|^{p}=\lambda^{p}\left|\xi'-\frac{\xi}{\lambda}\right|^{p}.
\]
Since $|\xi'|\ge r>0$ and $|\xi|\le R$, we have, uniformly for large
$\lambda$,
\[
\left|\xi'-\frac{\xi}{\lambda}\right|^{p}=|\xi'|^{p}\left(1-\frac{2}{\lambda}\frac{\xi'\cdot\xi}{|\xi'|^{2}}+\frac{1}{\lambda^{2}}\frac{|\xi|^{2}}{|\xi'|^{2}}\right)^{p/2}.
\]
The function $z\mapsto(1+z)^{p/2}$ is analytic near $z=0$, and the
quantity
\[
-\frac{2}{\lambda}\frac{\xi'\cdot\xi}{|\xi'|^{2}}+\frac{1}{\lambda^{2}}\frac{|\xi|^{2}}{|\xi'|^{2}}
\]
is uniformly $O(\lambda^{-1})$. Therefore, for every $N\ge0$,
\[
\mathbf{E}_{\pi}|\xi-\lambda\xi'|^{p}=\lambda^{p}\sum^{N}_{n=0}a^{\pi}_{n,\mathrm{sp}}\lambda^{-n}+O(\lambda^{p-N-1}),
\]
where the first two coefficients are
\[
a^{\pi}_{0,\mathrm{sp}}=\mathbf{E}|\xi'|^{p},
\]
and
\[
a^{\pi}_{1,\mathrm{sp}}=-p\,\mathbf{E}_{\pi}\left(|\xi'|^{p-2}\xi'\cdot\xi\right).
\]
Now consider the temporal term. We have
\[
|\theta-\lambda^{2}\theta'|^{q}=\lambda^{p}\left|\theta'-\frac{\theta}{\lambda^{2}}\right|^{q}.
\]
The difficulty is that $\theta'$ is not bounded away from $0$. However,
in virtue of Schilder’s theorem \cite{dembo2009large}, $\theta'$
has an exponentially small lower tail, hence finite negative moments
of every order. We use this to justify the expansion at the level
of expectations. On the set
\[
A_{\lambda}:=\left\{ \theta\le\frac{1}{2}\lambda^{2}\theta'\right\} ,
\]
we may write
\[
\left|\theta'-\frac{\theta}{\lambda^{2}}\right|^{q}=(\theta')^{q}\left(1-\frac{\theta}{\lambda^{2}\theta'}\right)^{q}.
\]
Taylor expansion gives, for every $K\ge0$,
\[
(\theta')^{q}\left(1-\frac{\theta}{\lambda^{2}\theta'}\right)^{q}=\sum^{K}_{k=0}\binom{q}{k}(-1)^{k}\theta^{k}(\theta')^{q-k}\lambda^{-2k}+\mathcal{R}_{K,\lambda}.
\]
The remainder satisfies
\[
|\mathcal{R}_{K,\lambda}|\le C_{K}\lambda^{-2K-2}\theta^{K+1}(\theta')^{q-K-1}.
\]
This is integrable, because $\theta$ has all positive moments and
$\theta'$ has all negative moments. On the complement $A^{c}_{\lambda}$,
we have
\[
\theta>\frac{1}{2}\lambda^{2}\theta'.
\]
For any $m>0$,
\[
\mathbf{1}_{A^{c}_{\lambda}}\le2^{m}\lambda^{-2m}\theta^{m}(\theta')^{-m}.
\]
Since all the corresponding positive and negative moments are finite,
the contribution of $A^{c}_{\lambda}$ is $O(\lambda^{-M})$ for any
prescribed $M$, by choosing $m$ sufficiently large. Hence the temporal
term also has an expectation-level expansion:
\[
\mathbf{E}_{\pi}|\theta-\lambda^{2}\theta'|^{q}=\lambda^{p}\sum^{K}_{k=0}\binom{q}{k}(-1)^{k}\mathbf{E}_{\pi}\left[\theta^{k}(\theta')^{q-k}\right]\lambda^{-2k}+O(\lambda^{p-2K-2}).
\]
Combining the spatial and temporal expansions, we obtain, for every
$N\ge0$,
\[
T^{\pi}_{p}(\lambda)^{p}=\lambda^{p}\left(S_{p}+\sum^{N}_{n=1}a^{\pi}_{n}\lambda^{-n}+O(\lambda^{-N-1})\right),
\]
where
\[
a^{\pi}_{1}=-p\,\mathbf{E}_{\pi}\left(|\xi'|^{p-2}\xi'\cdot\xi\right).
\]
Since $S_{p}>0$, we may take the $p$-th root. Namely,
\[
T^{\pi}_{p}(\lambda)=\lambda S^{1/p}_{p}\left(1+\sum^{N}_{n=1}\frac{a^{\pi}_{n}}{S_{p}}\lambda^{-n}+O(\lambda^{-N-1})\right)^{1/p}.
\]
The function $x\mapsto(1+x)^{1/p}$ is analytic near $x=0$, so expanding
it gives
\[
T^{\pi}_{p}(\lambda)=\lambda S^{1/p}_{p}+\sum^{N}_{j=0}b^{\pi}_{j}\lambda^{-j}+O(\lambda^{-N-1}).
\]
The constant term comes from $a^{\pi}_{1}$:
\[
b^{\pi}_{0}=\frac{a^{\pi}_{1}}{pS^{1-1/p}_{p}}=-\frac{\mathbf{E}_{\pi}\left(|\xi'|^{p-2}\xi'\cdot\xi\right)}{S^{1-1/p}_{p}}.
\]
\end{proof}

\begin{proof}
(Proposition \ref{thm:scaling optimal}) Let 
\[
(\xi,\theta)\sim\text{Law}(Z_{\tau_{U}},\tau_{U}).
\]
By the scaling property of Brownian motion, 
\[
(\lambda\xi,\lambda^{2}\theta)\sim\text{Law}(Z_{\tau_{\lambda U}},\tau_{\lambda U}).
\]
Hence the scaling coupling gives 
\[
\Lambda^{p}_{p}(U,\lambda U)\le|\lambda-1|^{p}\mathbf{E}(|\xi|^{p})+|\lambda^{2}-1|^{\frac{p}{2}}\mathbf{E}(\theta^{\frac{p}{2}}).
\]
We now show that this coupling is optimal. For the spatial part, let
$Y$ be any random variable with the law of $\lambda\xi$ (defined
on the same sample space $\Omega$) . By the reverse triangular inequality
\[
\|Y-\xi\|_{L^{p}(\Omega)}\ge\bigl|\|Y\|_{L^{p}(\Omega)}-\|\xi\|_{L^{p}(\Omega)}\bigr|=|\lambda-1|\,\|\xi\|_{L^{p}(\Omega)}.
\]
That is, 
\[
\mathbf{E}(|Y-\xi|^{p})\ge|\lambda-1|^{p}\mathbf{E}(|\xi|^{p}).
\]
Equality is attained by taking $Y=\lambda\xi$. For the temporal part,
the same phenomenon occurs since $q:=\frac{p}{2}\ge1$. Since the
same scaling coupling simultaneously attains the optimal spatial and
temporal costs, it is optimal for the joint cost, which proves the
formula. 
\end{proof}

The difference with the case $p\ge2$ is that the temporal exponent
\[
q:=\frac{p}{2}
\]
now satisfies $0<q<1$. Consequently, the $L^{q}$-argument no longer
applies, since $L^{q}$ is not a normed space for $q<1$, and one
cannot use the reverse triangle inequality in $L^{q}$. Equivalently,
the one-dimensional cost $|t-s|^{q}$ is no longer convex. Thus, while
the scaling coupling still provides a natural and explicit admissible
transport plan, its optimality is no longer automatic, and in general
we only obtain the above upper bound rather than an exact formula.
If $1\le p<2$, we still have 
\[
\Lambda_{p}(U,\lambda U)^{p}\le|\lambda-1|^{p}\mathbf{E}(|Z_{\tau_{U}}|^{p})+|\lambda^{2}-1|^{\frac{p}{2}}\mathbf{E}(\tau^{\frac{p}{2}}_{U}).
\]

\begin{proof}
(Proposition \ref{prop:Unit disc case}) For Brownian motion started
at $0$ in the unit disk, the exit point and exit time are independent.
More precisely, we may write
\[
Z_{\tau_{\mathbb{D}}}=\xi=e^{i\Theta},
\]
where $\Theta$ is uniformly distributed on $[0,2\pi)$ and is independent
of
\[
\theta:=\tau_{\mathbb{D}}.
\]
By Brownian scaling, the exit pair from $\lambda\mathbb{D}$ has the
same law as
\[
(\lambda\xi',\lambda^{2}\theta'),
\]
where
\[
\xi'=e^{i\Theta'},
\]
$\Theta'$ is uniformly distributed on $[0,2\pi)$, $\theta'$ has
the same law as $\theta$, and $\Theta'$ is independent of $\theta'$.
We first compute the supremum. The spatial part satisfies the pointwise
bound
\[
|e^{i\Theta}-\lambda e^{i\Theta'}|\le1+\lambda.
\]
Equality holds if and only if
\[
e^{i\Theta'}=-e^{i\Theta}.
\]
Thus the maximal spatial contribution is
\[
(1+\lambda)^{p}.
\]
For the time part, since $p\geq2$, the exponent $\frac{p}{2}\geq1$.
The one-dimensional rearrangement inequality gives that the maximal
coupling of $\theta$ and $\lambda^{2}\theta'$ for the cost
\[
|t-s|^{q}
\]
is the anti-monotone quantile coupling:
\[
\theta=Q(r),\qquad\lambda^{2}\theta'=\lambda^{2}Q(1-r).
\]
See \cite{CSS1976}. Since the exit angle and exit time are independent
in the disk, the maximizing spatial coupling and the maximizing time
coupling can be combined. Therefore
\[
\Phi_{p}(\mathbb{D},\lambda\mathbb{D})^{p}=(1+\lambda)^{p}+\int^{1}_{0}\left|Q(r)-\lambda^{2}Q(1-r)\right|^{q}\,dr.
\]
Now consider the independent coupling. Let
\[
\xi=e^{i\Theta},\qquad Y=\lambda e^{i\Theta'},
\]
where $\Theta,\Theta'$ are independent uniform random variables.
Then
\[
\Theta-\Theta'
\]
is uniform modulo $2\pi$, and hence
\[
\mathbf{E}(|\xi-Y|^{p})=\frac{1}{2\pi}\int^{2\pi}_{0}\left(1+\lambda^{2}-2\lambda\cos\theta\right)^{p/2}\,d\theta=I_{p}(\lambda).
\]
Similarly, for the independent time variables,
\[
\mathbf{E}(|\theta-\lambda^{2}\theta'|^{q})=\int^{1}_{0}\int^{1}_{0}\left|Q(r)-\lambda^{2}Q(u)\right|^{q}drdu=J_{q}(\lambda).
\]
Thus
\[
T^{\mathrm{ind}}_{p}(\mathbb{D},\lambda\mathbb{D})^{p}=I_{p}(\lambda)+J_{q}(\lambda).
\]
\end{proof}

\begin{proof}
(Proposition \ref{prop:supremal and independent disc}) Set
\[
A:=\Phi_{p}(\mathbb{D},\lambda\mathbb{D}),\qquad B:=T^{\mathrm{ind}}_{p}(\mathbb{D},\lambda\mathbb{D}),\qquad q:=\frac{p}{2}.
\]
We first prove the corresponding powered gap. By the formulas for
$\Phi_{p}$ and $T^{\mathrm{ind}}_{p}$,
\[
A^{p}-B^{p}=\left[(1+\lambda)^{p}-I_{p}(\lambda)\right]+\left[\int^{1}_{0}|Q(r)-\lambda^{2}Q(1-r)|^{q}\,dr-J_{q}(\lambda)\right].
\]
The second bracket is nonnegative, because for the one-dimensional
cost
\[
(x,y)\longmapsto|x-\lambda^{2}y|^{q},\qquad q\ge1,
\]
the countermonotone coupling maximizes the transport cost among all
couplings of the law of $\tau_{\mathbb{D}}$ with itself. Hence
\[
A^{p}-B^{p}\ge(1+\lambda)^{p}-I_{p}(\lambda).
\]
We now estimate the spatial term. Recall that
\[
I_{p}(\lambda)=\frac{1}{2\pi}\int^{2\pi}_{0}(1+\lambda^{2}-2\lambda\cos\theta)^{p/2}\,d\theta.
\]
Let
\[
f(x):=(1+\lambda^{2}-2\lambda x)^{p/2},\qquad-1\le x\le1.
\]
Since $p\ge2$, the function $f$ is convex. Therefore, for $x\in[-1,1]$,
\[
f(x)\le\frac{1+x}{2}f(1)+\frac{1-x}{2}f(-1).
\]
Taking $x=\cos\theta$ and integrating over $[0,2\pi]$, we obtain
\[
I_{p}(\lambda)\le\frac{f(1)+f(-1)}{2}=\frac{(\lambda-1)^{p}+(1+\lambda)^{p}}{2}.
\]
Thus
\[
A^{p}-B^{p}\ge\frac{(1+\lambda)^{p}-(\lambda-1)^{p}}{2}.
\]
Since $A^{p}-B^{p}\ge0$, we have $A\ge B$. By convexity of $x\mapsto x^{p}$,
\[
A^{p}-B^{p}\le pA^{p-1}(A-B).
\]
Hence
\[
A-B\ge\frac{A^{p}-B^{p}}{pA^{p-1}}\ge\frac{(1+\lambda)^{p}-(\lambda-1)^{p}}{2pA^{p-1}}.
\]
It remains to bound $A$ from above. From the formula for $\Phi_{p}$,
\[
A^{p}=(1+\lambda)^{p}+\int^{1}_{0}|Q(r)-\lambda^{2}Q(1-r)|^{q}\,dr.
\]
By Minkowski's inequality in $L^{q}(0,1)$,
\[
\left(\int^{1}_{0}|Q(r)-\lambda^{2}Q(1-r)|^{q}\,dr\right)^{1/q}\le(1+\lambda^{2})\mathbf{E}(\tau^{q}_{\mathbb{D}})^{1/q}.
\]
Since $\lambda\ge1$, we have
\[
(1+\lambda^{2})^{q}\le(1+\lambda)^{p}.
\]
Thus
\[
A^{p}\le(1+\lambda)^{p}\left(1+\mathbf{E}(\tau^{p/2}_{\mathbb{D}})\right).
\]
Hence
\[
A^{p-1}\le(1+\lambda)^{p-1}\left(1+\mathbf{E}(\tau^{p/2}_{\mathbb{D}})\right)^{(p-1)/p}.
\]
Substituting this into the previous inequality gives
\[
\Phi_{p}(\mathbb{D},\lambda\mathbb{D})-T^{\mathrm{ind}}_{p}(\mathbb{D},\lambda\mathbb{D})=A-B\ge\frac{1}{p\left(1+\mathbf{E}(\tau^{p/2}_{\mathbb{D}})\right)^{(p-1)/p}}
\]
where we used 
\[
\frac{(1+\lambda)^{p}-(\lambda-1)^{p}}{(1+\lambda)^{p-1}}\geq2.
\]
\end{proof}

\section{Conclusion and further directions }\label{sec:Conclusion-and-further}

The results of this paper show that $p$-Brownian convergence provides
a natural probabilistic way to compare planar domains through the
joint law of the Brownian exit position and exit time. After recording
several geometric features of this convergence, we introduced Brownian
transport costs between domains and showed that the resulting Brownian-Wasserstein
distance metrizes $p$-Brownian convergence. We then compared different
natural couplings of exit pairs, in particular the same-path and independent
couplings, and obtained explicit criteria showing when one is cheaper
than the other. The analysis becomes especially transparent for $p=2$,
where the costs can be expressed in terms of mean exit times and intersections
of exit times. We also derived exact formulas for homothetic domains,
proved monotonicity properties for the independent cost, and computed
explicit examples such as strips and discs. These results suggest
that the competition between couplings reflects meaningful geometric
information about the domains, leading naturally to the problem of
understanding the Brownian coupling threshold for families of translated
or reflected domains.

\subsubsection*{Coupling transition threshold}

Inspired by the overlapping strips, we formulate a more general problem.
Let $U$ be a bounded domain, and call $U^{s}$ the symmetric counterpart
of $U$ with respect to the $y$-axis. Set 
\[
U_{a}=U+a
\]
the shifted copy of $U$, where $a\in\mathbb{R}$. Let $I=(s,t)$
be the largest open interval of parameters $a$ such that $0\in U_{a}$.
Our question is the sign of the function 
\[
\iota:a\in I\longmapsto T^{\mathrm{ind}}_{2}(U_{a},(U_{a})^{s})-T^{\mathrm{same}}_{2}(U_{a},(U_{a})^{s}).
\]
If $U=U^{s}$ then it is clear that $\iota$ is even and $\iota(0)>0$.
We conjecture that if $U$ is, in addition, convex then $\iota$ switch
sign once on the positive part of $I$. More precisely, there exists
a unique $a^{*}\in I$ such that $\iota(a^{*})=0$ and 
\[
\iota_{|(0,a^{*})}>0,\,\,\iota_{|(a^{*},t)}<0.
\]
 We would call such an $a^{*}$ \emph{the coupling transition threshold}.
It marks the transition between the regime where the same-path coupling
is cheaper and the regime where the independent coupling is cheaper.
The overlapping-strip model (introduced at the end of section \ref{sec:A-Brownian-Wasserstein-distance})
gives an explicit prototype for this transition. Let
\[
\mathcal{S}_{\ell}=\left\{ z=x+iy:\ |x|<\frac{\ell}{2}\right\} ,\qquad U_{a}=\mathcal{S}_{\ell}+a.
\]
Then
\[
(U_{a})^{s}=\mathcal{S}_{\ell}-a.
\]
Writing
\[
\varepsilon=\frac{\ell}{2}-a,\qquad L=\frac{\ell}{2}+a
\]
we obtain
\[
U_{a}=\{-\varepsilon<x<L\},\qquad(U_{a})^{s}=\{-L<x<\varepsilon\}.
\]
Thus the overlapping strips correspond to the translation-reflection
model with
\[
\frac{L}{\varepsilon}=\frac{\frac{\ell}{2}+a}{\frac{\ell}{2}-a}.
\]
Keep labeling the exit time from $U_{a}$ as $\theta$ and the exit
time from $(U_{a})^{s}$ as $\theta'$ with $\theta\independent\theta'$.
The quantity $\mathbf{E}(\theta\wedge\theta')$ admits an explicit
spectral representation. Set
\[
\alpha:=\frac{\varepsilon}{L+\varepsilon},
\]
then, by applying the reflection principle infinitely many times at
the boundary of a strip ($U_{a}$ in our case), we get 
\[
\begin{alignedat}{1}\mathbf{P}(\theta>t) & =\frac{4}{\pi}\sum^{\infty}_{n=0}\frac{1}{2n+1}\sin\left((2n+1)\pi\alpha\right)\exp\left(-\frac{(2n+1)^{2}\pi^{2}}{2\ell^{2}}t\right)\\
 & =\frac{4}{\pi}\sum^{\infty}_{n=0}\frac{1}{2n+1}\sin\left((2n+1)\pi\alpha\right)\exp\left(-\frac{(2n+1)^{2}\pi^{2}}{2(L+\varepsilon)^{2}}t\right).
\end{alignedat}
\]
For the trick of reflecting the Brownian path infinitely many times,
see \cite{schilling2014brownian,markowsky2018distribution}. On the
other hand, as $\theta\independent\theta'$ we obtain
\[
\mathbf{E}(\theta\wedge\theta')=\int^{+\infty}_{0}\mathbf{P}(\theta>t)^{2}dt.
\]
Termwise integration gives the rapidly convergent double series
\[
\mathbf{E}(\theta\wedge\theta')=\frac{32\ell^{2}}{\pi^{4}}\sum^{\infty}_{m,n=0}\frac{\sin\!\left((2m+1)\pi\alpha\right)\sin\!\left((2n+1)\pi\alpha\right)}{(2m+1)(2n+1)\bigl((2m+1)^{2}+(2n+1)^{2}\bigr)}.
\]
Now put
\[
\rho:=\frac{L}{\varepsilon}.
\]
Then
\[
\ell=\varepsilon(\rho+1),\qquad\alpha=\frac{1}{\rho+1}.
\]
Thus

\[
\mathbf{E}(\theta\wedge\theta')=\varepsilon^{2}F(\rho),
\]
where

\[
F(\rho):=\frac{32(\rho+1)^{2}}{\pi^{4}}\sum^{\infty}_{m,n=0}\frac{\sin\!\left(\frac{(2m+1)\pi}{\rho+1}\right)\sin\!\left(\frac{(2n+1)\pi}{\rho+1}\right)}{(2m+1)(2n+1)\bigl((2m+1)^{2}+(2n+1)^{2}\bigr)}.
\]
Consequently,
\[
T^{\mathrm{same}}_{2}(U_{a},(U_{a})^{s})<T^{\mathrm{ind}}_{2}(U_{a},(U_{a})^{s})\quad\Longleftrightarrow\quad F\left(\frac{L}{\varepsilon}\right)<3.
\]
Now recall 
\[
\iota(a):=T^{\mathrm{ind}}_{2}(U_{a},(U_{a})^{s})-T^{\mathrm{same}}_{2}(U_{a},(U_{a})^{s}).
\]
Hence

\[
\iota(a)=\varepsilon\left[\sqrt{6\rho-2F(\rho)}-\sqrt{6(\rho-1)}\right].
\]
As 
\[
\varepsilon=\frac{\ell}{2}-a,\qquad\rho=\frac{\frac{\ell}{2}+a}{\frac{\ell}{2}-a}
\]
we get 
\[
\iota(a)=\left(\frac{\ell}{2}-a\right)\left[\sqrt{6\frac{\frac{\ell}{2}+a}{\frac{\ell}{2}-a}-2F\left(\frac{\frac{\ell}{2}+a}{\frac{\ell}{2}-a}\right)}-\sqrt{6\left(\frac{\frac{\ell}{2}+a}{\frac{\ell}{2}-a}-1\right)}\right].
\]
In particular,

\[
\text{sign}(\iota(a))=\text{sign}\left(3-F\left(\frac{\frac{\ell}{2}+a}{\frac{\ell}{2}-a}\right)\right).
\]
Thus the transition threshold is determined by
\[
F(\rho^{*})=3,\qquad a^{*}=\left(\frac{\rho^{*}-1}{\rho^{*}+1}\right)\frac{\ell}{2}.
\]
This means that 
\[
\frac{a^{*}}{\frac{\ell}{2}}=\frac{\rho^{*}-1}{\rho^{*}+1}
\]
is constant for any $\ell>0$. 

\begin{figure}[H]

\begin{centering}
\includegraphics[width=11cm,totalheight=11cm,keepaspectratio]{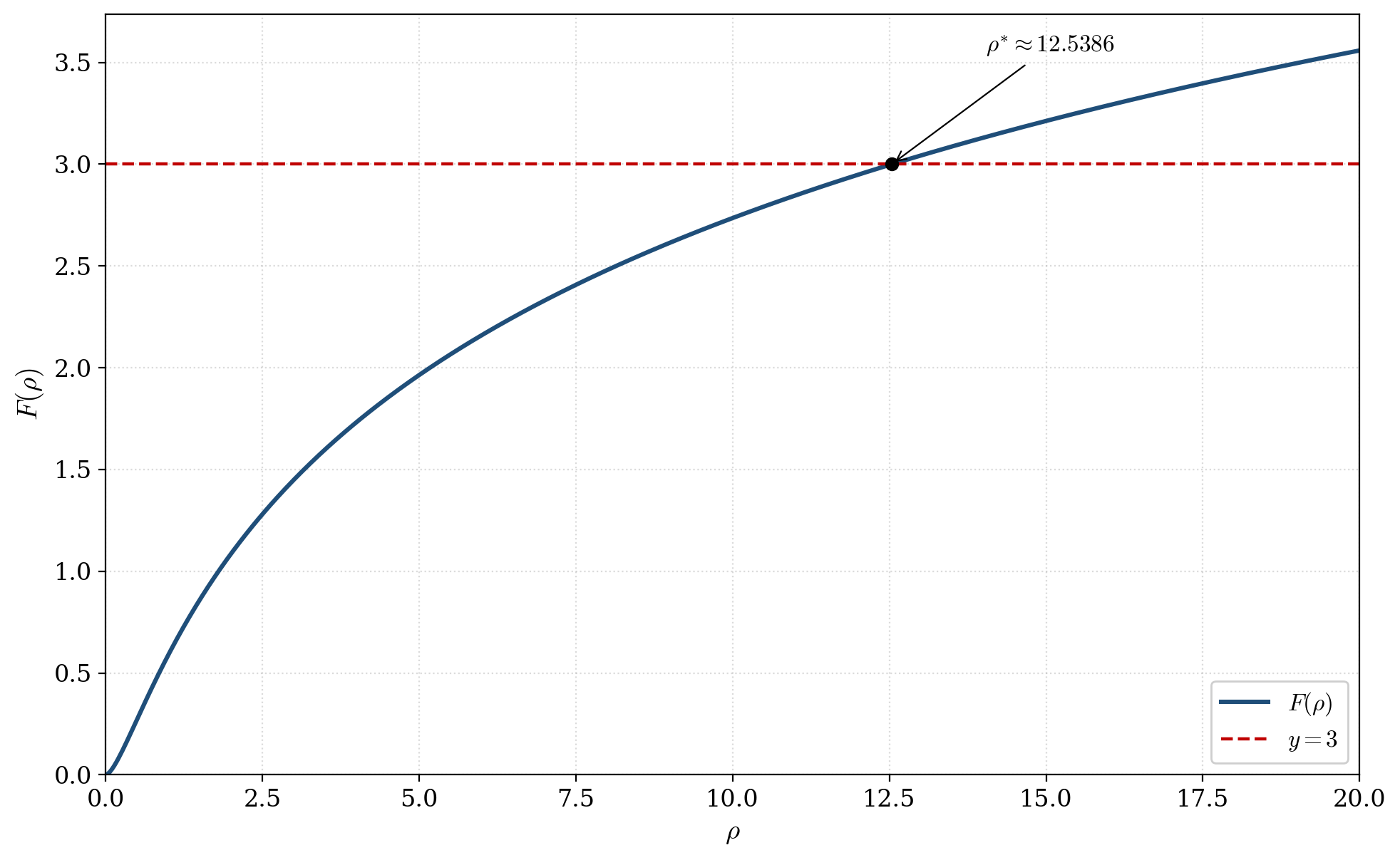}\caption{Plot of the function $\rho\protect\mapsto F(\rho)$.}
\par\end{centering}
\end{figure}

As $\rho^{*}\approx12.5386$ then our coupling transition threshold
is
\[
a^{*}\approx0.85227\cdot\frac{\ell}{2},
\]
which corresponds to an overlap width of 
\[
2\varepsilon^{*}=0.14772\ell
\]
Thus same-path remains cheaper until the overlap has shrunk to about
14.8\% of the strip width.

\subsubsection*{Brownian norm of a domain $U$}

Let $U$ be a domain. The quantity 
\begin{equation}
\Vert U\Vert^{\mathbf{B}}_{p}:=S^{1/p}_{p}=\left(\mathbf{E}(|Z_{\tau_{U}}|^{p})+\mathbf{E}(\tau^{p/2}_{U})\right)^{1/p}.\label{eq:Brownian norm}
\end{equation}
encodes the size of $U$ as seen a Brownian particle wandering inside
$U$. We  call it as Brownian norm of $U$. In particular, whenever
$U\subset V$ and both quantities are finite, one has
\begin{equation}
U\subset V\Longrightarrow\Vert U\Vert^{\mathbf{B}}_{p}\leq\Vert V\Vert^{\mathbf{B}}_{p}.\label{norm increa}
\end{equation}
Moreover, for every $\alpha\in\mathbb{C}$, 
\[
\Vert\alpha U\Vert^{\mathbf{B}}_{p}=|\alpha|\Vert U\Vert^{\mathbf{B}}_{p}.
\]
On the set of analytic maps acting on the unit disc, we can similarly
consider the map
\[
f\longmapsto\Vert f\Vert^{\mathbf{B}}_{p}=\Vert f(\mathbb{D})\Vert^{\mathbf{B}}_{p}.
\]
When $p=2$, $\Vert U\Vert^{\mathbf{B}}_{2}$ simplifies to 
\[
\Vert U\Vert^{\mathbf{B}}_{2}=\sqrt{3\mathbf{E}(\tau_{U})}.
\]
The case $p=2$ is special because as $\mathbf{E}(\tau_{U})$ is the
value at the origin of the torsion function of $U$\cite{philippin1996isoperimetric}.
We now compute the Brownian norm for a non-simply connected non-circular
domain. Fix $c>0$ and $0<\alpha<\gamma<\beta$. For $\eta>0$, let
\[
E_{\eta}:=\left\{ (x,y)\in\mathbb{R}^{2}:\frac{x^{2}}{c^{2}\cosh^{2}\eta}+\frac{y^{2}}{c^{2}\sinh^{2}\eta}<1\right\} .
\]
The ellipses $\partial E_{\eta}$ have the same foci $(\pm c,0)$.
Set
\[
d:=c\cosh\gamma,
\]
and define the shifted elliptic annulus
\[
U:=\left\{ x+iy:\frac{(x+d)^{2}}{c^{2}\cosh^{2}\beta}+\frac{y^{2}}{c^{2}\sinh^{2}\beta}<1\right\} \setminus\overline{\left\{ x+iy:\frac{(x+d)^{2}}{c^{2}\cosh^{2}\alpha}+\frac{y^{2}}{c^{2}\sinh^{2}\alpha}<1\right\} }.
\]
\begin{center}
\begin{tikzpicture}[
    >={Stealth[length=2.4mm]},
    scale=1.5,
    every node/.style={font=\small}
  ]
 
  \def\cVal{1}        
  \def\alphaVal{0.4}  
  \def\gammaVal{0.9}  
  \def\betaVal{1.3}   
 
  \pgfmathsetmacro{\dVal}{\cVal*cosh(\gammaVal)}             
  \pgfmathsetmacro{\aBeta} {\cVal*cosh(\betaVal)}
  \pgfmathsetmacro{\bBeta} {\cVal*sinh(\betaVal)}
  \pgfmathsetmacro{\aAlpha}{\cVal*cosh(\alphaVal)}
  \pgfmathsetmacro{\bAlpha}{\cVal*sinh(\alphaVal)}
  \pgfmathsetmacro{\aGamma}{\cVal*cosh(\gammaVal)}
  \pgfmathsetmacro{\bGamma}{\cVal*sinh(\gammaVal)}
 
  \begin{scope}[even odd rule]
    \clip (-\dVal,0) ellipse ({\aBeta} and {\bBeta})
          (-\dVal,0) ellipse ({\aAlpha} and {\bAlpha});
    \fill[blue!12] (-5,-2.5) rectangle (1.5,2.5);
  \end{scope}
 
  \draw[->, gray!60] (-4.0,0) -- (1.0,0) node[right, black] {$x$};
  \draw[->, gray!60] (0,-2.0) -- (0,2.0)   node[above, black] {$y$};
 
  \draw[dashed, gray!70, thick]
        (-\dVal,0) ellipse ({\aGamma} and {\bGamma});
 
  \draw[very thick]
        (-\dVal,0) ellipse ({\aBeta} and {\bBeta});
 
  \draw[very thick]
        (-\dVal,0) ellipse ({\aAlpha} and {\bAlpha});
 
  \pgfmathsetmacro{\Fxr}{-\dVal + \cVal}
  \pgfmathsetmacro{\Fxl}{-\dVal - \cVal}
  \filldraw (\Fxr,0) circle (0.9pt);
  \filldraw (\Fxl,0) circle (0.9pt);
  \node[below=2pt] at (\Fxr,0) {\scriptsize $-d{+}c$};
  \node[below=2pt] at (\Fxl,0) {\scriptsize $-d{-}c$};
 
  \filldraw (-\dVal,0) circle (0.7pt);
  \node[above=2pt] at (-\dVal,0) {\scriptsize $(-d,0)$};
 
  \filldraw[red!75!black] (0,0) circle (1.4pt);
  \node[red!70!black, above right=1pt] at (0,0) {$0$};
 
  \node[above]                at (-\dVal,\bBeta+0.05)         {$\partial E_{\beta}$};
  \node[above, gray!55!black] at (-\dVal,\bGamma+0.05)        {$\partial E_{\gamma}$};
  \node[above]                at (-\dVal,\bAlpha+0.05)        {$\partial E_{\alpha}$};
 
  \node at (-2.7,1.4) {\large $U$};
 
\end{tikzpicture}
\end{center}
Solving the torsion problem in $U$ using elliptic coordinates simplifies
to solving standard differential equations (we do not include the
details here). We get the closed form 
\[
\|U\|^{\mathbf{B}}_{2}=\left\{ \frac{3c^{2}}{4}\left[\frac{\beta-\gamma}{\beta-\alpha}\cosh(2\alpha)+\frac{\gamma-\alpha}{\beta-\alpha}\cosh(2\beta)-\cosh(2\gamma)+\frac{\sinh(2(\beta-\gamma))+\sinh(2(\gamma-\alpha))}{\sinh(2(\beta-\alpha))}-1\right]\right\} ^{1/2}.
\]

\subsubsection*{Brownian repulsion}

Besides the minimal Brownian transport cost, it is natural to consider
the opposite extremal quantity. We define the $p$-Brownian repulsion
of $U$ by
\[
\mathbf{R}_{p}(U):=\Phi_{p}(U,U).
\]
Thus $\mathbf{R}_{p}(U)$ measures the maximal possible separation
between two Brownian exit pairs having the same law. Unlike the Brownian-Wasserstein
distance, it does not vanish when the two domains coincide. It is
rather a self-interaction quantity, or a Brownian analogue of the
diameter of the exit-pair distribution. It is also controlled by the
Brownian norm
\[
0\le\mathbf{R}_{p}(U)\le2\|U\|^{\mathbf{B}}_{p}.
\]
Indeed, for any coupling of two exit pairs $(\xi,\theta)$ and $(\xi',\theta')$
with law $\mu_{U}$, the triangle inequality for the Brownian-parabolic
metric gives
\[
d_{p}((\xi,\theta),(\xi',\theta'))\le d_{p}((\xi,\theta),(0,0))+d_{p}((\xi',\theta'),(0,0)).
\]
Taking $L^{p}$-norms and then the supremum over couplings gives the
bound. If $U$ is symmetric with respect to the origin, namely $U=-U$,
then
\[
(Z_{\tau_{U}},\tau_{U})\stackrel{d}{=}(-Z_{\tau_{U}},\tau_{U}).
\]
Therefore the coupling
\[
(\xi',\theta')=(-\xi,\theta)
\]
is admissible, and yields the lower bound
\[
\mathbf{R}_{p}(U)^{p}\ge2^{p}\,\mathbf{E}(|Z_{\tau_{U}}|^{p}).
\]
In particular, when $p=2$,

\[
\mathbf{R}_{2}(U)^{2}\ge4\,\mathbf{E}(|Z_{\tau_{U}}|^{2})=8\,\mathbf{E}(\tau_{U}).
\]
Together with the general upper bound
\[
\mathbf{R}_{2}(U)^{2}\le4\|U\|^{2}_{\mathbf{B},2}=12\,\mathbf{E}(\tau_{U}),
\]
we obtain, for every origin-symmetric domain with finite mean exit
time,
\[
8\,\mathbf{E}(\tau_{U})\le\mathbf{R}_{2}(U)^{2}\le12\,\mathbf{E}(\tau_{U}).
\]
For the unit disc, by Proposition \ref{prop:Unit disc case}, repulsion
can be written explicitly as 
\[
\mathbf{R}_{p}(\mathbb{D})^{p}=2^{p}+\int^{1}_{0}|Q(r)-Q(1-r)|^{p/2}\,dr.
\]

\bibliographystyle{plain}
\bibliography{NumericalApproach}

\end{document}